\documentclass[11pt, leqno]{amsart}

\usepackage[T1]{fontenc}

\usepackage{amsfonts, delarray, amssymb, amsmath, amsthm, a4, a4wide}
\usepackage{thmtools, cite}

\usepackage{mathrsfs}
\usepackage{comment}
\usepackage{mathtools}
\usepackage{graphicx}
\usepackage{latexsym}
\usepackage{epsfig}
\usepackage{color}

\usepackage[plainpages=false,hypertexnames=false,pdfpagelabels]{hyperref}
\definecolor{medium-blue}{rgb}{0,0,.8}

\hypersetup{colorlinks, linkcolor={blue}, citecolor={medium-blue}, urlcolor={medium-blue}}

\hypersetup{hidelinks}

\numberwithin{equation}{section}

\DeclareMathOperator{\osc}{\text{osc }}
\DeclareMathOperator{\diam}{\text{diam}}
\DeclareMathOperator{\diver}{\text{div}}
\DeclareMathOperator{\dist}{\text{dist}}
\newcommand{\e}{\varepsilon}
\newcommand{\p}{\partial}
\newcommand{\F}{\mathbf{F}}
\def\R{\mathbb{R}}

\newtheorem{thm}{Theorem}[section]

\newtheorem{lem}[thm]{Lemma}
\newtheorem{prop}[thm]{Proposition}

\theoremstyle{definition}
\newtheorem{defn}[thm]{Definition}

\newtheorem{rem}[thm]{Remark}

\makeatletter
\@namedef{subjclassname@2020}{
  \textup{2020} Mathematics Subject Classification}
\makeatother
\begin{document} 

\title[Green's function approach to linearized Monge--Amp\`ere equations ]{A Green's function approach to linearized Monge--Amp\`ere equations in divergence form and application to singular Abreu type equations}
\author{Chong Gu}
\address{Department of Mathematics, Indiana University,
Bloomington, IN 47405, USA}
\email{chongu@iu.edu}
\author{Nam Q. Le}
\address{Department of Mathematics, Indiana University,
Bloomington, IN 47405, USA}
\email{nqle@iu.edu}
\thanks{The authors were supported in part by the National Science Foundation under grant DMS-2452320.}


\begin{abstract}
In this paper, we establish local and global regularity estimates for linearized Monge--Amp\`{e}re equations in divergence form via 
critical Lorentz space estimates for the Green's function of the linearized Monge--Amp\`{e}re operator and its gradient. These estimates hold under suitable conditions on the data and 
the convex Monge--Amp\`ere potential is assumed to have Hessian determinant bounded between two positive constants.  
As an application, we obtain the solvability in all dimensions of the second boundary value problem for a class of singular fourth-order Abreu type equations that arise from the approximation analysis of variational problems subject to convexity constraints.
\end{abstract}
\maketitle

\section{Introduction}
In this paper, we establish local and global  regularity estimates for solutions to linearized Monge--Amp\`{e}re equations in divergence form
\begin{equation}\label{eqdiv}
	L_u v: = D_i(U^{ij}D_j v) = \diver \F +\mu,
\end{equation}
where $\F$ is a vector field and $\mu$ is a signed Radon measure, via critical Lorentz space estimates for the Green's function, also known as the fundamental solution, of the linearized Monge--Amp\`{e}re operator 
\[
    L_u:= D_i(U^{ij} D_j).
\]
Throughout, \[U = (U^{ij})_{1\leq i, j\leq n}:= (\det D^2 u) (D^2 u)^{-1}\] is the cofactor matrix of the Hessian matrix $D^2 u$ of a convex Monge--Amp\`ere potential $u\in C^3(\Omega)$, where the function $u$ satisfies 
\begin{equation}\label{MAu}
	0 < \lambda \leq \det D^2 u \leq \Lambda \quad \text{ in }\Omega\subset\R^n,
\end{equation}
where $\lambda$ and $\Lambda$ are constants.
We always assume $n\geq 2$ and repeated indices are summed.  More precise assumptions on $ \F$ and $\mu$ will be given in corresponding theorems; see also Remark \ref{Fmurem}.

\medskip
Due to the divergence-free property of $U$, that is, $D_iU^{ij} = 0$ for all $j$, the operator $L_u$ can be rewritten in nondivergence form:
\[L_u =D_j (U^{ij}D_i)= U^{ij}D_{ij}.\] 
The coefficient matrix $U$ of $L_u$ arises from linearizing the Monge-Amp\`ere operator $\det D^2 u$.
One can also note that $L_u v$ is the coefficient of $t$ in the expansion
$$\det D^2 (u + t v) =\det D^2 u + t\, U^{ij} D_{ij}v + \cdots+ t^n \det D^2 v.$$

Under the assumption \eqref{MAu},  the cofactor matrix $U$ is positive definite, so \eqref{eqdiv} is an elliptic equation. However, it is in general degenerate and/or singular since the eigenvalues of $U$ could tend to zero and/or infinity.  This makes its analysis challenging.

\medskip
Caffarelli and Guti\'{e}rrez \cite{CG97} established a fundamental interior Harnack inequality for the homogeneous linearized Monge--Amp\`{e}re equation
$U^{ij}D_{ij} v=0$. Their result is an affine invariant version of the classical Harnack inequality of Moser, and Krylov and Safonov for elliptic equations in divergence form and nondivergence form, respectively. 
Interior H\"older estimates consequently follow. Central in their analysis is the geometry of sections of the Monge--Amp\`ere potential function $u$ which replace Euclidean balls. Sections are defined as sublevel sets of convex functions after subtracting their supporting hyperplanes.
Since then, many developments and applications have been obtained by many authors including \cite{CWZ, GN1, GN2, K, KLWZ, Le_Manu, Le_TAMS, Le_CMP, Le_Harnack, Le20, Le24, LN3, LS1, Mal13, Mal14, Mal17, S0, TiW, TW00, TW08, Wang25}.
We refer the reader to \cite{Le24} for an overview and will restrict ourselves to works closely related to the subject of this paper.

\medskip
Equations of the type \eqref{eqdiv} arise in several contexts such as the semigeostrophic equations in meteorology \cite{ACDF, BB, CF, CNP, Le_CMP, Loe05} and singular Abreu equations in the calculus of variations with a convexity constraint \cite{CR, KLWZ, Le_CPAM, Le23, LZ}. 

\medskip
Let us mention some works related to the case of $\mu=0$ under \eqref{MAu}. Loeper \cite{Loe05} established H\"older estimates for  \eqref{eqdiv} using integral information on $v$ under the assumption that $\det D^2 u$ is close to a constant.
When $n = 2$ and  \eqref{MAu} is satisfied, the second author  \cite{Le_CMP, Le20} obtained interior and global H\"{o}lder estimates for \eqref{eqdiv} assuming $\F \in L^\infty(\Omega; \R^n) \cap W^{1, n}_{\text{loc}}(\Omega; \R^n)$.
With the additional assumption  $D^2 u\in L^{1+\varepsilon_n}(\Omega)$ for some $\varepsilon_n > (n+1)(n-2)/2$, these estimates were extended to
higher dimensions in \cite[Chapter 15]{Le24}. This assumption holds in dimension two in view of the $W^{2, 1+\varepsilon}$ estimates for the Monge--Amp\`{e}re equation established by
De Philippis--Figalli--Savin \cite{DPFS13} and Schmidt \cite{Sch}.
In view of Caffarelli's $W^{2, p}$ estimates \cite{Caff90} and Wang's counterexamples \cite{W95} for the Monge--Amp\`{e}re equation, the above integrability condition may be thought of as an assumption on $\Lambda/\lambda$ in \eqref{MAu}. Such higher integrability of $D^2 u$ was required for the Moser iteration.
Later, by assuming $(D^2 u)^{1/2}\F\in L^q(\Omega;\R^n)$ for some $q>n$ and using De Giorgi's iteration, Wang \cite{Wang25} 
established interior H\"{o}lder estimates for \eqref{eqdiv} relying on the $L^\infty$ norm of the solution $v$, while the second author's results only rely on the $L^p$ norm of $v$ for $p\in (1, +\infty)$. For all  bounded vector field $\F$, Wang's condition essentially requires
$D^2 u$ to be in $L^r$ where $r>n/2$. In this case, Wang's H\"older estimates were established via Moser's iteration technique by Kim \cite{K} who also studied equations with drifts. 

\medskip
Very recently, when $\F=0$, Cui--Wang--Zhou \cite{CWZ} use potential theory and Campanato type estimates to study equation \eqref{eqdiv} for signed Radon measures $\mu$ whose total variation on each compactly contained section of $u$ of height $h$ grows like $h^{\frac{n-2}{2}+\e}$ for some $\e>0$. In particular, they obtained interior H\"older estimates for  \eqref{eqdiv} under  \eqref{MAu} when $\mu=\diver \F_0$ is a negative Radon measure and $\F_0$ is a bounded vector field. Remarkably, these results give new interior higher-order estimates for singular Abreu equations.
In all these works \cite{CWZ, K, Wang25}, the Monge--Amp\`ere Sobolev inequality \cite{Le_CMP,TiW} plays an important role.

\medskip
We will give here a unified, different proof of the interior H\"older estimates in \cite{CWZ, Wang25}. Moreover, we also obtain a global version. Our approach uses fine properties of the Green's function of the linearized Monge--Amp\`{e}re operator $L_u$ and avoids the  Monge--Amp\`ere Sobolev inequality. In the context of linearized Monge--Amp\`{e}re  equations, our approach has its root in the previous work of Nguyen and the second author \cite{LN3} where global H\"older estimates were obtained using $L^q$ norm of the right-hand side where $q>n/2$.
Crucial to our approach are 
uniform estimates 
for solutions to 
\eqref{eqdiv} with zero boundary value. Our approach gives uniform estimates assuming only $(D^2 u)^{1/2}\F$ being in the Lorentz space $L^{n, 1}$ when $n\geq 3$.

\medskip
As an application, we obtain the solvability of the second boundary value problem for singular Abreu type equations, some of which
 in dimensions at least three are not accessible by previous approaches. 
 
\medskip
The analysis of the Green's function of the linearized Monge--Amp\`{e}re operator starts with the work of Tian--Wang \cite{TiW} and then Maldonado \cite{Mal13, Mal14, Mal17} and the second author \cite{Le_Manu, Le_TAMS, Le_CMP, Le20, Le24}. 
Let $g_S(\cdot, y)$ be the Green's function of $L_u$ in $S\subset \Omega$ with pole $y \in S$. 
Building on properties of the Green's function put together in \cite{Le24}, we will establish interior and global estimates for $(D^2u)^{-1/2} D_x g_S(\cdot, y)$ in the weak $L^{n/(n-1)}$ space, also known as the Lorentz space $L^{n/(n-1),\infty}$, when $n\geq 3$. This space is critical as can be seen for the case of $u(x)=|x|^2/2$ where $L_u$ now becomes the Laplace operator, and its Green's function does not belong to smaller Lorentz spaces. Our results are affine analogues of those in Gr\"uter--Widman \cite{GW} for uniformly elliptic equations in dimensions $n\geq 3$. We have an extra logarithmic factor for Lorentz space type estimates in dimension two. These estimates build upon  Lorentz space type estimates for the Green's function $g_S$ itself; in this case, we have more refined analysis as we have a good control on the shapes of the superlevel sets of $g_S$. 

\medskip
The quantity $(D^2u)^{-1/2} D_x g_S(\cdot, y)$ was first studied by Maldonado \cite{Mal13, Mal17} as a natural Monge--Amp\`{e}re gradient in the sense of the Monge--Amp\`{e}re Sobolev inequality obtained in \cite{Le_CMP, TiW}. Our estimates for $(D^2u)^{-1/2} D_x g_S(\cdot, y)$ in strong $L^p$ spaces where $p\in (1, n/(n-1))$ follow from  Lorentz space estimates and they recover and extend results established in \cite{Mal17} in the case of sections compactly contained in the domain. 
As applications, we will use these estimates to obtain local and global uniform and  H\"older estimates for solutions to \eqref{eqdiv}  under suitable assumptions.

\medskip
Before stating our main results, 
we introduce some relevant concepts. 

\begin{defn}[Sections]
	Let $u\in C^1(\Omega)\cap C(\overline{\Omega})$ be a convex function and $h > 0$. 
If $x_0 \in \overline \Omega$, then the \textit{section of $u$ centered at $x_0$ with height $h$} is defined by 
	           \[
	               S_u(x_0, h):= \big\{x \in \overline{\Omega}: u(x) < u(x_0) + Du(x_0)\cdot (x - x_0) + h\big\}.
	           \]
	           In the case of $x_0\in\p\Omega$, we require that $u$ is differentiable at $x_0$.
\end{defn}

\begin{defn}[Green's function of the linearized Monge--Amp\`{e}re operator] 
\label{Gdefn}
Let $\Omega$ be a bounded convex domain in $\R^n$ and $u\in C^3(\Omega)$ be a convex function satisfying \eqref{MAu}. Assume $V \subset  \Omega$ is open. 
  Let $\delta_y$ be the Dirac measure giving the unit mass to $y$. 
Then, for each $y \in V$, there exists a unique function $g_V(\cdot, y): V \rightarrow[0, \infty]$ with the following properties:
    \begin{enumerate}
        \item[(a)] $g_V(\cdot, y) \in W_0^{1, q}(V) \cap W^{1,2}\left(V \backslash B_r(y)\right)$ for all $q<\frac{n}{n-1}$ and all $r>0$.
        \item[(b)] $g_V(\cdot, y)$ is a weak solution of
            \[
                    -D_i\left(U^{i j} D_j g_V(\cdot, y)\right)  =\delta_y \quad \text { in } V, \quad
	               g_V(\cdot, y)  =0 \quad \text { on } \p V,
            \]
            that is, denoting $D_jg_V(x, y) = D_{x_j}g_V(x, y) = \frac{\p}{\p x_j}g_V(x, y)$, we have
	           \begin{equation}\label{Gdef}
	            \int_V U^{i j} D_j g_V(x, y) D_i \psi(x) \; dx=\psi(y) \quad \text{ for all } \psi \in C_c^{\infty}(V).
	           \end{equation}
             \end{enumerate}
  
    We call $g_V(\cdot, y)$ the \textit{Green's function} of the linearized Monge--Amp\`{e}re operator $L_u=D_i\left(U^{i j} D_j\right)$ in $V$ with \textit{pole} $y$. We set $g_V(y, y) = +\infty$.
\end{defn}
\begin{defn}[Proper sets]
We call a nonempty open set $V\subset\R^n$ {\it proper} if $V$ satisfies an exterior cone condition at every boundary point. Examples of such proper sets include intersections of a bounded convex domain $\Omega\subset\R^n$ with sections of a convex function $u\in C^1(\overline{\Omega})$ or open balls.
\end{defn}
\begin{defn}[Lorentz spaces]
Let $\Omega \subset \R^n$ be bounded and open, and $f: \Omega \longrightarrow \R$ be a Lebesgue-measurable function. 
For $1 \leq p < \infty$ and $0 < q \leq \infty$, we define
\[
    \|f\|_{L^{p, q}(\Omega)}:=
    \begin{cases}
        p^{\frac{1}{q}}\Big(\int_0^{\infty} t^q |\{x\in \Omega: |f(x)| > t\}|^{\frac{q}{p}} \frac{d t}{t}\Big)^{\frac{1}{q}} &\text{ if } q<\infty,\\
        \sup _{t>0} t |\{x\in \Omega: |f(x)| > t\}|^{1/p}  &\text{ if } q=\infty.
    \end{cases}
\] 
The \textit{Lorentz space $L^{p, q}(\Omega)$} consists of all Lebesgue-measurable functions $f$ defined on $\Omega$ such that $\|f\|_{L^{p, q}(\Omega)} < \infty$. The Lorentz space $L^{p, \infty}$ coincides with the \textit{weak $L^p$ space}, and the Lorentz space $L^{p, p}$ is the usual $L^p$ space. 
\end{defn}

\medskip
 {\bf Notation}. We use $c=(\ast,\ldots,\star)$ and $C= C(\ast,\ldots,\star)$ to denote positive constants $c, C$ depending on the quantities appearing in the parentheses; they may change from line to line. We use $D_i=\p/\p_{x_i}$ and $D_x f(x, y)$ to denote the gradient of $f$ in the $x$-variable.
 For a Lebesgue measurable set $E \subset \R^n$,   $|E|$ denotes its $n$-dimensional Lebesgue measure. We use $\mathcal{H}^{n-1}$ to denote the $(n-1)$-dimensional Hausdorff measure.
 For a locally integrable function $\mu:\R^n\to\R$, 
 we can view it as a signed Radon measure and denote $|\mu|(A) = \int_A |\mu|\, dx$ for any Lebesgue measurable set $A\subset\R^n$.

\medskip

The rest of this paper will be organized as follows. We state our main results in Section \ref{sec_res}. In Section \ref{sec: pre}, we recall some background materials for our analysis
and state a key estimate for the Monge--Amp\`ere gradient of the Green's function. 
In Section \ref{sec: int_G}, we will prove Theorems \ref{Green_int_thm} and \ref{Green_gl_thm} on Lorentz space estimates for the Green's function. In Section \ref{sec_max}, we establish maximum principles for  linearized Monge--Amp\`ere equation \eqref{eqdiv}. 
We prove an interior Harnack inequality in Theorem \ref{int_Har} and interior H\"older estimates for solutions to  \eqref{eqdiv} in Section \ref{sec_Har}. The proof of Theorem \ref{glb_holder} on global H\"older estimates will be given in Section \ref{sec_glH}.
We present an application to the solvability of singular Abreu equations in Section \ref{sec_Ab}.

\section{Statement of the main results}
\label{sec_res}
In this section, we state our main results and give some brief comments on them. 

\subsection{Lorentz space estimates for the Green's function} 
Our first main result establishes Lorentz space estimates for the Green's function 
of the linearized Monge--Amp\`{e}re operator in compactly contained sections.

\begin{thm}\label{Green_int_thm}
	Let $u\in C^3(\Omega)$ be a convex function satisfying \eqref{MAu}, where $\Omega\subset \R^n$. Assume $S_u(x_0, 2h) \Subset \Omega$ where $x_0\in\Omega$ and $h>0$. Let $g_{S_h}(\cdot, y)$ be the Green's function of the linearized Monge--Amp\`{e}re operator $D_i(U^{ij}D_j)$ in $S_u(x_0, h)$ with pole $y \in S_u(x_0, h)$. Then, for $t > 0$, 
    \begin{equation}\label{int_Dg_level}
    \begin{split}
        |\{x \in S_u(x_0, h) : |(D^2u(x))^{-1/2}D_xg_{S_u(x_0, h)}(x,y)| > t\}|\\  \leq 
        \begin{cases}
             Ct^{-2}\big(h+ \log \max\{t, 1\}\big) &\text{ if } n = 2,\\
             Ct^{-\frac{n}{n-1}} &\text{ if } n \geq 3,
        \end{cases}
        \end{split}
    \end{equation}
    where $C = C(n, \lambda, \Lambda) > 0$. 
    
    Consequently, for all $y \in S_u(x_0, h)$, 
    \begin{enumerate}
        \item[$\bullet$] if $n\geq 3$, then
        \begin{equation}\label{int_Dg_1}
            \|(D^2u)^{-1/2}D_xg_{S_u(x_0, h)}(\cdot,y)\|_{L^{\frac{n}{n-1}, \infty}(S_u(x_0, h))} \leq C(n, \lambda, \Lambda);
        \end{equation}    
        \item[$\bullet$] if $n \geq 2$ and $p \in \big( 1, \frac{n}{n-1}\big)$, then 
        \begin{equation}\label{int_Dg_2}
            \int_{S_u(x_0, h)} |(D^2 u(x))^{-1/2} D_x g_{S_u(x_0, h)}(x, y)|^p \; dx \leq C(n, \lambda, \Lambda, p)h^{\frac{n}{2} - \frac{n-1}{2}p}.     
        \end{equation}
    \end{enumerate}
\end{thm}

Next, we state global analogues of estimates in Theorem \ref{Green_int_thm} under suitable assumptions for which Savin's boundary localization theorem for the Monge--Amp\`ere equation \cite[Theorem 3.1]{S1} is applicable. 

\medskip
{\bf Global structural assumptions.} Let $\Omega \subset \mathbb{R}^n$ be a convex domain and assume that there exists a constant $\rho>0$ such that
\begin{equation}\label{glb_1}
     \Omega \subset B_{1 / \rho}(0) \subset \mathbb{R}^n
\end{equation}
and for each $y \in \p \Omega$, 
\begin{equation}
   \text{there is an interior ball } B_\rho(z) \subset \Omega \text{ such that } y \in \p B_\rho(z). 
\end{equation}
Let $u \in C^{1,1}(\overline{\Omega}) \cap C^3(\Omega)$ be a convex function satisfying
\begin{equation}
    0<\lambda \leq \det D^2 u \leq \Lambda \quad \text { in } \Omega.
\end{equation}
Assume further that on $\p \Omega, u$ separates quadratically from its tangent hyperplanes; namely, for all $x_0, x \in \p \Omega$, we have
\begin{equation}\label{glb_4}
    \rho\left|x-x_0\right|^2 \leq u(x)-u\left(x_0\right)-D u\left(x_0\right) \cdot\left(x-x_0\right) \leq \rho^{-1}\left|x-x_0\right|^2.
\end{equation}

\medskip
\begin{thm}\label{Green_gl_thm}
	Assume that $u$ and $\Omega$ satisfy \eqref{glb_1}--\eqref{glb_4}. Let $V\subset\Omega$ be open, proper, and contained in a section $S_u(x_0, h)$ of $u$ of height $h>0$. Let $g_V(\cdot, y)$ be the Green's function of the linearized Monge--Amp\`{e}re operator $D_i(U^{ij} D_j)$ in $V$ with pole $y \in V$. Then, for  $t >0$, 
    \begin{equation}\label{glb_D_lvl}
        |\{x \in V : |(D^2 u(x))^{-1/2} D_x g_V(x, y)| > t\}| \leq 
    \begin{cases}
         Ct^{-2} (h + \log \max\{t, 1\}) &\text{ if } n=2, \\
         Ct^{-\frac{n}{n-1}} &\text{ if } n \geq 3,
    \end{cases}
    \end{equation}
    where $C = C(n, \lambda, \Lambda, \rho) > 0$. 
    
    Consequently, for all $y\in V$,
    \begin{enumerate}
   \item[$\bullet$] if $ n\geq 3$, then
        \begin{equation}\label{glb_wLp}
        	\|(D^2 u)^{-1/2} D_x g_V(\cdot, y)\|_{L^{\frac{n}{n-1}, \infty}(V)} \leq C(n, \lambda, \Lambda, \rho).
        \end{equation}
       \item[$\bullet$] if $n\geq 2$ and $p\in \big( 1, \frac{n}{n-1}\big)$, then
        \begin{equation}\label{glb_D_Lp}
		\int_V |(D^2 u(x))^{-1/2} D_x g_V(x, y)|^p \; dx \leq C(n, \lambda, \Lambda, \rho, p) h^{\frac{n}{2} - \frac{n-1}{2}p}.  
	    \end{equation}
         \end{enumerate}
\end{thm}

\medskip
Under the assumptions \eqref{glb_1}--\eqref{glb_4}, by \cite[Lemma 9.7]{Le24}, there exists $M=M(n,\lambda,\Lambda,\rho)>0$ such that \[\overline{\Omega}\subset S_u(x_0, M) \quad\text{for all}\quad x_0\in\overline{\Omega}.\] Therefore, $\Omega$ is a section of $u$ with height comparable to $1$ and Theorem \ref{Green_gl_thm} is applicable to $V=\Omega$. 

\subsection{Applications of Lorentz space estimates to the linearized Monge--Amp\`{e}re equations} 
Applying Theorem 
\ref{Green_int_thm} and the Caffarelli--Guti\'{e}rrez Harnack inequality (Theorem \ref{thm: CG_Harnack}), we obtain a Harnack inequality for nonnegative solutions to \eqref{eqdiv}.

\begin{thm} \label{int_Har}
    Let $\Omega \subset \R^n$ and $u\in C^3(\Omega)$ be a convex function satisfying \eqref{MAu}. Let $\F \in  W^{1, n}_{\text{loc}}(\Omega; \R^n)$ and $\mu\in L^n_{\text{loc}}(\Omega)$. Suppose that $S_u(x_0, 2h) \Subset \Omega$ where $x_0\in\Omega$ and $h>0$. Let $v\in W^{2, n}_{\text{loc}}(S_u(x_0, h))\cap C(\overline{S_u(x_0, h)})$ be a nonnegative solution to
    \[U^{ij}D_{ij} v = \diver \F +\mu \quad\text{in} \quad S_u(x_0, h).\]
    \begin{enumerate}
        \item[(i)] Assume $ n\geq 3$, $\mu=0$, and $ (D^2u)^{1/2} \F \in L^{n, 1}(S_u(x_0, h);\R^n)$. Then 
        \begin{equation}\label{int_H_wLp}
            \sup_{S_u(x_0, h/2)} v \leq C(n, \lambda, \Lambda) \Big(  \inf_{S_u(x_0, h/2)} v + 
                \|(D^2u)^{1/2} \F \|_{L^{n, 1}(S_u(x_0, h))}\Big).
        \end{equation}
        \item[(ii)] Assume that 
        $(D^2u)^{1/2} \F \in L^{q}(S_u(x_0, h);\R^n)$ for some $q > n$ and
             there exist $M_0\geq 0$ and  $\e>0$ such that
  \begin{equation}\label{muMH}|\mu|(S_u(z, s)) \leq M_0 s^{\frac{n-2}{2} +\e}\quad \text{for all sections } S_u(z, s)\Subset S_u(x_0, 2h).
         \end{equation}
         Then
        \begin{equation}\label{int_H_Lp}
            \sup_{S_u(x_0, h/2)} v \leq C(n, \lambda, \Lambda) \Big(  \inf_{S_u(x_0, h/2)} v + 
                C_\ast \|(D^2u)^{1/2} \F \|_{L^q(S_u(x_0, h))}h^{\frac{q-n}{2q}}+ C_\star M_0 h^{\e} \Big),
        \end{equation}
        where $C_\ast = C_\ast(n, \lambda, \Lambda, q)>0$ and $C_\star = C_\star(n, \lambda, \Lambda, \e)>0$.
          \end{enumerate}    
\end{thm}

\medskip

Next, we discuss H\"{o}lder continuity of solutions to \eqref{eqdiv}. Theorem \ref{int_Har} allows us to give a new proof of the interior H\"{o}lder estimates for \eqref{eqdiv} when 
$(D^2u)^{1/2} {\bf F} \in L^q(\Omega;\R^n)$ for some $q > n$ and $\mu\in L^n(\Omega)$ satisfying \eqref{muMH}.
These estimates are stated in Theorem \ref{thm: int}.  
  They were first proved by Wang \cite[Theorem 1.5]{Wang25} and Cui--Wang--Zhou \cite[Theorem 1.2]{CWZ}. Moreover, we are able to obtain the following global H\"{o}lder estimates.

\begin{thm}\label{glb_holder}
	Let $\Omega$ be a uniformly convex domain in $\R^n$, that is, for all $z\in \partial \Omega$, there is a ball $B_R(z_0)$ such that $\Omega \subset B_R(z_0)$ and $\partial \Omega \cap \partial B_R(z_0) = \{z\}$ for some uniform convexity radius $R>0$.  Let $\partial \Omega \in C^3$ and $u\in C^{1, 1}(\overline{\Omega})\cap C^3(\Omega)$ be a convex function satisfying \eqref{MAu} and $u|_{\partial \Omega} \in C^3$. Assume $\F \in W^{1, n}(\Omega; \R^n)$, $\mu\in L^n(\Omega)$, and $\phi \in C^\alpha(\partial \Omega)$ for some $\alpha \in (0, 1)$. Let $v \in  W^{2, n}_{\text{loc}}(\Omega)\cap C(\overline{\Omega})$ be the solution to 
	\[
		U^{ij}D_{ij} v = \diver \F + \mu \quad\text{ in } \Omega, \;\quad v = \phi\quad \text{ on } \partial \Omega.
	\]
	Assume $(D^2u)^{1/2} {\bf F} \in L^q(\Omega;\R^n)$ for some $q > n$, and
	there exist
	$M_0\geq 0$ and $\e>0$ such that
  \begin{equation}
  \label{muglthm}
  |\mu|(S_u(z, s)) \leq M_0 s^{\frac{n-2}{2} +\e}\quad \text{for all sections } S_u(z, s)\subset \overline{\Omega}.\end{equation}
Then, there exists $\beta = \beta(n, \lambda, \Lambda, q, \e, \alpha)\in (0, 1)$ such that
	\[
		\|v\|_{C^{\beta}(\overline{\Omega})}\leq 
		C \big(\|\phi\|_{C^\alpha(\partial \Omega)} + \| (D^2u)^{1/2} \F\|_{L^q(\Omega)} + M_0\big),
	\]
	where $C > 0$ depends only on $n, \lambda, \Lambda, q, \e, \alpha, \|u\|_{C^3(\partial \Omega)}$, $R$, and the $C^3$ regularity of $\partial \Omega$.
  \end{thm}
As an application of Theorem \ref{glb_holder}, we obtain in Section \ref{sec_Ab} the solvability of the second boundary value problem for singular fourth-order Abreu type equations, some of which
 in dimensions at least three are not accessible by previous approaches. 

\subsection{Comments} We briefly comment on our results, assumptions, and methods of the proofs.

\begin{rem} Some remarks on the results are in order.
\begin{enumerate}
\item Under assumption \eqref{MAu}, when $y= x_0$
 and the integral over the whole section $S_u(x_0, h)$ in inequality \eqref{int_Dg_2} is replaced by one over $S_u(x_0, h/2)$,  
 Maldonado \cite[Theorem 1.1]{Mal17} 
 established the following estimate  by a different method:
 \[\int_{S_u(x_0, h/2)} |(D^2 u(x))^{-1/2} D_x g_{S_u(x_0, h)}(x, x_0)|^p \; dx \leq C(n, \lambda, \Lambda, p)h^{\frac{n}{2} - \frac{n-1}{2}p}.    \]

\item When $n \geq 3$, the exponent $\frac{n}{2} - \frac{n-1}{2}p$ in \eqref{int_Dg_2} comes from applying \eqref{int_Dg_1} and the H\"{o}lder inequality in Lorentz spaces together with the volume estimates for sections. For the same reason, the exponent
$\frac{q-n}{2q}$ in \eqref{int_H_Lp} comes from \eqref{int_H_wLp}.
\item 
It would be interesting to remove the logarithmic terms in Theorems \ref{Green_int_thm} and \ref{Green_gl_thm} in dimension two.

\item    By applying Theorem \ref{Green_gl_thm} and the boundary Harnack inequality for the linearized Monge--Amp\`ere equation \cite{Le_TAMS}, we can obtain a boundary analogue of Theorem \ref{int_Har}.
		\end{enumerate}
\end{rem}
\begin{rem} 
\label{Fmurem}
We comment briefly on some assumptions.
\begin{enumerate}
\item In Theorems \ref{int_Har} and \ref{glb_holder}, we assume the vector field $\bf{F}$ to be in $W^{1, n}$ and the measure $\mu$ to be an $L^n$ function
	only to use the representation formula \eqref{G_rep} involving the Green's function, $\diver \F$, and $\mu$. However, our estimates will not depend on this regularity of $\F$ and $\mu$. In particular, they do not depend on
	$\|\mu\|_{L^n}$ in Theorem \ref{glb_holder}. Note that
	\[|\mu|(S_u(z, s)) \leq \|\mu\|_{L^n(\Omega)}|S_u(z, s)|^{\frac{n-1}{n}} \leq C(n,\lambda) \|\mu\|_{L^n(\Omega)} s^{\frac{n-1}{2}}. \]
\item	As will be seen in applications (see Lemma \ref{mu_bound} and Theorem \ref{Abthm}), when $\mu=\diver\F_0$ for a vector field $\F_0\in W^{1, n}(\Omega;\R^n)\cap L^\infty(\Omega;\R^n)$ such that $\mu$ has a definite sign,
	$M_0$ will be chosen to depend on $\|\F_0\|_{L^\infty(\Omega)}$ rather than on $\|\diver \F_0\|_{L^n(\Omega)}$. This is the case of $\F_0$ being the gradient of a convex or concave function.
 We do not pursue the issue of finding  optimal regularity conditions on $\F$ and $\mu$ in this paper.
	\end{enumerate}
\end{rem}
\begin{rem} We comment briefly on the methods of the proofs. 
\begin{enumerate}
\item The key idea in the proofs of Theorems \ref{Green_int_thm} and \ref{Green_gl_thm} is to 
estimate the distribution function for the Monge--Amp\`ere gradient $(D^2u)^{-1/2}Dg$ of the Green's function $g$ of the linearized Monge--Amp\`{e}re operator  via estimates of the distribution function of 
$g$ (see Proposition \ref{keyprop}) and then optimize. This allows us to take advantage of results for the distribution function of $g$ that have been developed so far.  Though simple, this idea can be applicable to more general elliptic operators in divergence form.
\item 
In the proofs of Theorems \ref{int_Har} and \ref{glb_holder}, we follow the framework in \cite[Chapter 15]{Le24} for linearized Monge--Amp\`{e}re equations in divergence form
and use the Caffarelli--Guti\'{e}rrez Harnack inequality,  Savin's boundary localization theorem, and maximum principles (such as Propositions \ref{intmax} and \ref{thm: glb_mp}) to estimate solutions with prescribed boundary values.
\item In the course of proving the maximum principles such as \eqref{int_ubd1}, instead of working with $D_x g_{S_h}(\cdot, y) \cdot \F$, we work with 
$(D^2 u)^{-1/2} D_x g_{S_h}(\cdot, y) \cdot (D^2 u)^{1/2} \F$. The integrability of the Monge--Amp\`ere gradient of the Green's function in 
Theorems \ref{Green_int_thm} and \ref{Green_gl_thm} explains the assumed integrability of $(D^2 u)^{1/2} \F$ and this seems to be optimal without further restrictions. On the other hand, when $\mu$ is a Radon measure with certain measure growth of the total variation $|\mu|$ on sections, the layer-cake formula can be used to obtain uniform estimates; see \eqref{int_mu3}. For this, our refined control on the shape of the superlevel sets of the Green's function proves to be crucial. 
	\end{enumerate}
\end{rem}

\section{Preliminaries and a key estimate}\label{sec: pre}
In this section, we recall some properties of sections of solutions to the Monge--Amp\`{e}re equation and the Green's functions of the linearized Monge--Amp\`{e}re operator and 
state a key estimate for the Monge--Amp\`ere gradient of the Green's function.

\medskip
We will frequently use the following volume estimates; see \cite[Lemmas 5.6 and 5.8]{Le24}. 
\begin{lem}[Volume estimates for sections]\label{vol_est}
Let $u\in C^2(\Omega)$ be a convex function satisfying \eqref{MAu} on a bounded convex domain $\Omega \subset \R^n$.  Let $x\in\Omega$ and $h>0$.  Then 
\[|S_u(x, h)| \leq C(n) \lambda^{-1 / 2} h^{n / 2}.\]
If  $S_u(x, h)\Subset \Omega$, then 
\[
	c(n) \Lambda^{-1 / 2} h^{n / 2} \leq |S_u(x, h)|,
\]
where $c, C$ are positive constants depending only on $n$.
\end{lem}
Note that if $S_u(x_0, h)\Subset\Omega$, then the convexity of $u$ implies that $S_u(x_0, h)$ is a bounded convex domain and $u(x)= u(x_0) + Du(x_0)\cdot (x-x_0) + h$ on $\p S_u(x_0, h)$.

\medskip
We also need the following geometric property of sections; see \cite[Theorem 5.30]{Le24}.

\begin{thm}[Inclusion property of interior sections]
\label{thm: inclusion}
Let $u\in C^2(\Omega)$ be a convex function satisfying \eqref{MAu}. Then, there exist constants $c_0(n, \lambda, \Lambda)>0$ and $p_1(n, \lambda, \Lambda) \geq 1$ with the following property. Assume $S_u\left(x_0, 2 t\right) \Subset \Omega$ and $0<r<s \leq 1$.
 If $x_1 \in S_u\left(x_0, r t\right)$, then
        \[
            S_u\left(x_1, c_0(s-r)^{p_1} t\right) \subset S_u\left(x_0, s t\right).
        \]
\end{thm}

We collect here some facts on the Green's function of the linearized Monge--Amp\`{e}re operator; see \cite[Chapter 14]{Le24} for more details.

\begin{rem}\label{rem: G_prop}
	We will use the following properties of the Green's functions in the setting of Definition \ref{Gdefn}.
	\begin{enumerate}
		\item By approximation arguments, we can use $\psi \in W_0^{1,2}(V) \cap C(\overline{V})$ as test functions to \eqref{Gdef}.
		\item (Representation formula) Assume that $V$ is proper. 
		If $\varphi \in L^n(V)$, then there exists a unique solution $\psi \in W_{\text{loc}}^{2, n}(V) \cap W_0^{1,2}(V) \cap C(\overline{V})$ to
		 	\[
		 		-U^{ij} D_{ij} \psi = \varphi \quad\text{ in } V \quad \text{ and } \psi = 0 \quad\text{ on } \partial V. 
		 	\]
			Use $\psi$ as a test function to \eqref{Gdef}, then the following holds:
			\begin{equation}\label{G_rep}
				\psi(y)=\int_V g_V(\cdot, y) \varphi \; dx.
			\end{equation}	
	\end{enumerate}
\end{rem}

We also need the following regularity of the Green's function away from the pole.

\begin{prop}\label{prop: G_reg}
Let $u$, $\Omega$, and $V$ satisfy one of the following sets of conditions:
\begin{enumerate}
	\item $\Omega$ is bounded convex domain in $\R^n$, $u \in C^3(\Omega)$ is a convex function satisfying \eqref{MAu}, and $V\Subset \Omega$ is open. 
	\item  $u$ and $\Omega$ satisfy \eqref{glb_1}--\eqref{glb_4}, and $V \subset  \Omega$ is open. 
\end{enumerate}
Let $g_V(\cdot, x_0)$ be the Green's function of $L_u=D_i\left(U^{i j} D_j\right)$ in $V$ with pole $x_0 \in V$.
\begin{enumerate}
	\item[$\bullet$] If $E \Subset V \backslash\left\{x_0\right\}$, then
$g_V(\cdot, x_0) \in W_{\mathrm{loc}}^{2, n}(E)\cap W^{1,2}(E) \cap C(\overline{E})$ and $L_u g_V(\cdot, x_0)=0$ 
in $E$.
	\item[$\bullet$] Assume $V$ is proper. Then $g_V(\cdot, x_0) \in C\left(\overline{V} \setminus B_r\left(x_0\right)\right)$ for all $r>0$.
\end{enumerate}
\end{prop}

\medskip

We state a global H\"older gradient estimate for the Monge--Amp\`ere equation and verify \eqref{muglthm} under natural conditions that are suitable for applications.
\begin{lem} 
\label{mu_bound}
Assume that $u$ and $\Omega$ satisfy \eqref{glb_1}--\eqref{glb_4}. There exists $\alpha=\alpha(n,\lambda,\Lambda)\in (0, 1)$ such that the following statements hold.
	\begin{enumerate}
	\item[(i)] There exists $C_\ast=C_\ast(n, \lambda, \Lambda, \rho)>0$ such that 
	 \begin{equation}
\label{Duglb}
[Du]_{C^{\alpha}(\overline\Omega)}:=\sup_{x\neq y\in \overline\Omega}|Du(x)-Du(y)|/|x-y|^{\alpha}\leq C_\ast.
\end{equation}
\item[(ii)] Let $S_u(x_{0}, t_{0})$ be a section of $u$ with $x_0\in \overline{\Omega}$ and $t_{0}>0$. Then
\begin{equation}
\label{Hausdorffest}
\mathcal{H}^{n-1} (\p S_u(x_{0}, t_{0}))\leq C(n, \lambda, \Lambda,  \rho) t_0^{\frac{n-2}{2} + \frac{
\alpha}{\alpha +1}}.\end{equation}
Consequently, if $\F \in W^{1, n}(\Omega; \R^n) \cap L^{\infty}(\Omega; \R^n)$ and $\mu_{ \F}:=\diver \F$ has a definite sign, then
\begin{equation}
\label{muFb}
|\mu_{ \F}|(S_u(x_0, t_0)) \leq C(n, \lambda, \Lambda,  \rho) \|\F\|_{L^{\infty} (S_u(x_0, t_0))} t_0^{\frac{n-2}{2} + \frac{
\alpha}{\alpha +1}}.\end{equation}
	\end{enumerate}
\end{lem}
For the proof of Lemma \ref{mu_bound} (ii), we will use the following observation in Cui--Wang--Zhou \cite[Lemma 4.4]{CWZ}:
\begin{lem} 
\label{cwzlem}
Let $X\subset\R^n$ be a bounded convex domain containing an open ball $B_r(x_0)$. Then 
\[\mathcal{H}^{n-1}(\p X) \leq \frac{n|X|}{r}.\]
\end{lem}
In \cite[Lemma 4.4]{CWZ}, the lemma was proved for $\p X$ smooth. However, it is still valid for general bounded convex domain $X$. To see this, let $\{X_m\}_{m=1}^\infty$ be a sequence of uniformly convex domains with $C^{\infty}$ boundaries such that $\overline{X_m}$ converges to $\overline{X}$ in the Hausdorff distance; see \cite[Theorem 2.51]{Le24}.  Then applying \cite[Lemma 4.4]{CWZ} to $X_m$ and letting $m\to\infty$, we obtain the stated estimate.

\begin{proof}[Proof of Lemma \ref{mu_bound}] 
By the global $C^{1,\alpha}$ estimates for $u$ (see \cite[Theorem 9.5]{Le24}), there exist $\alpha(n, \lambda, \Lambda)\in (0, 1)$ and $C_\ast(n, \lambda, \Lambda, \rho)$ such that \eqref{Duglb} holds.

 From the mean value theorem, \eqref{Duglb} easily implies that
\begin{equation}
\label{Ballha}
B_{C_\ast^{-1/(1+\alpha)}h^{1/(1+\alpha)}}(x)\subset S_u(x, h)\quad \text{whenever } S_u(x, h)\Subset\Omega.
\end{equation}
By the dichotomy of sections in \cite[Proposition 9.8]{Le24}, one of the following is true: 
\begin{enumerate}
\item[(a)] $S_u(x_{0}, 2t_{0})\subset \Omega$. 
\item[(b)]  There exist $z\in\partial\Omega$ and  a constant  $\bar{c}(n,  \lambda, \Lambda,  \rho)>1$ 
such that
 $ S_u(x_{0}, 2t_{0})\subset S_u(z, \bar{c}t_{0}).$
 \end{enumerate}

\medskip
\noindent
Suppose (a) is true. 
Then, by Lemma \ref{cwzlem} and the volume estimates in Lemma \ref{vol_est}, we have
\[ \mathcal{H}^{n-1} (\p S_u(x_{0}, t_{0})) \leq \frac{n|S_u(x_0, t_0)|}{C_\ast^{-1/(1+\alpha)}t_0^{1/(1+\alpha)}}\leq \frac{C(n,\lambda) t_0^{n/2}}{C_\ast^{-1/(1+\alpha)}t_0^{1/(1+\alpha)}}=C(n,  \lambda, \Lambda,  \rho) t_0^{\frac{n-2}{2} + \frac{
\alpha}{\alpha +1}}.\]

\medskip
\noindent
Suppose (b) is true.
 If $\bar c t_0<c$ where $c= c(n,\lambda,\Lambda,\rho)>0$ is  small, then,  \cite[(9.11)]{Le24} gives
\[S_u(z, \bar c t_0) \subset B_{Ct_0^{1/2}|\log t_0|}(z).\]
By the monotonicity of the surface measure with respect to inclusion of convex  sets (see \cite[Lemma 2.71]{Le24}), we have 
 \begin{equation*}
  \mathcal{H}^{n-1} (\p S_u(x_{0}, t_{0}))
  \leq  \mathcal{H}^{n-1} (\p B_{Ct_0^{1/2}|\log t_0|}(z)) 
  \leq C(n,  \lambda, \Lambda,  \rho) t_0^{\frac{n-2}{2} + \frac{
\alpha}{\alpha +1}}.
\end{equation*}
 If $\bar c t_0>c$, then, again  by \cite[Lemma 2.71]{Le24}, we have 
 \[\mathcal{H}^{n-1} (\p S_u(x_{0}, t_{0})) \leq \mathcal{H}^{n-1} (\p \Omega) \leq \mathcal{H}^{n-1} (\p B_\rho(0)) \leq C(n)\rho^{n-1} \leq C(n,  \lambda, \Lambda,  \rho) t_0^{\frac{n-2}{2} + \frac{
\alpha}{\alpha +1}}.\]
We have established \eqref{Hausdorffest} in all cases.

To prove \eqref{muFb}, we can assume without loss of generality that $\mu_{ \F}\geq 0$. Let $\nu$ be the outer unit normal vector field on $\p S_u(x_0, t_0)$. Then, the divergence theorem gives
 \begin{equation*}
 \begin{split}
 |\mu_{ \F}|(S_u(x_0, t_0)) =\int_{S_u(x_0, t_0)} \diver \F \, dx&=\int_{\p S_u(x_0, t_0)} \F\cdot \nu\, d \mathcal{H}^{n-1}\\ 
 &\leq  \|\F\|_{L^{\infty} (S_u(x_0, t_0))} \mathcal{H}^{n-1} (\p S_u(x_{0}, t_{0}))\\& \leq C(n, \lambda, \Lambda,  \rho) \|\F\|_{L^{\infty} (S_u(x_0, t_0))} t_0^{\frac{n-2}{2} + \frac{
\alpha}{\alpha +1}}.
\end{split}
\end{equation*}
The lemma is proved.
\end{proof}
\medskip
For later references, we record here an easy consequence of the H\"older inequality and the volume estimates in Lemma \ref{vol_est}.
\begin{rem}
\label{n2rem}
Assume that $u$ and $\Omega$ satisfy \eqref{glb_1}--\eqref{glb_4} and $\mu\in L^q(\Omega)$ where $q>n/2$. Then, for all
$x_0\in \overline{\Omega}$ and $t_{0}>0$, we have
\[|\mu|(S_u(x_0, t_0))\leq C(n, q,  \lambda, \Lambda,  \rho) \|\mu\|_{L^q(S_u(x_0, t_0))} t_0^{\frac{n}{2}-\frac{n}{2q}}.\]
\end{rem}

\medskip 
We will use the following H\"{o}lder inequality in Lorentz spaces, due to O'Neil \cite{ONeil63}.
\begin{thm}[H\"{o}lder inequality in Lorentz spaces]
\label{thm: Hol_Lor}
Let $\Omega\subset\R^n$ be open.
Let $f, g: \Omega \longrightarrow \R$ be two Lebesgue-measurable functions. Suppose $1\leq p_1, p_2, p<\infty$ and $0<q_1, q_2, q \leq \infty$ satisfy $\frac{1}{p}=\frac{1}{p_1}+\frac{1}{p_2}$ and $\frac{1}{q}=\frac{1}{q_1}+\frac{1}{q_2}$. If $\|f\|_{L^{p_1, q_1}(\Omega)} < \infty$ and $\|g\|_{L^{p_2, q_2}(\Omega)} < \infty$, then
\[
\|f g\|_{L^{p, q}(\Omega)} \leq C(p_1, p_2, q_1, q_2)\|f\|_{L^{p_1, q_1}(\Omega)}\|g\|_{L^{p_2, q_2}(\Omega)}.
\]
    \end{thm}
We recall the layer-cake formula (see \cite[Lemma 2.75]{Le24}).
\begin{lem}[The layer-cake formula] 
\label{lcf}
Let $E\subset \R^n$ be a Borel set and let $v:E\rightarrow \R$ be a measurable function. Assume that $\mu:E\rightarrow\R$ is an integrable, nonnegative function. Let us also use $\mu$ for the measure with density $\mu(x)$. Then
$$\int_E |v(x)| \mu(x)\, dx= \int_0^\infty  \mu(E\cap \{|v|>t\})\,dt,$$
and for $p\in (1,\infty)$, we have
$$\int_E |v(x)|^p\, dx= p\int_0^\infty t^{p-1} |E\cap \{|v|>t\}|\, dt.$$
\end{lem}
\medskip
The following proposition is our key estimate. It allows us to estimate the distribution function for the Monge--Amp\`ere gradient of the Green's function $g$ via estimates of the distribution function of 
$g$.
\begin{prop}
\label{keyprop}
Let $\Omega\subset\R^n$ and let $V\subset\Omega$ be open and proper. Let $u\in C^3(\Omega)$ be a convex function satisfying $\det D^2 u\geq\lambda>0$.
 Let $g_V(\cdot, y)$ be the Green's function of the linearized Monge--Amp\`{e}re operator $D_i(U^{ij} D_j)$ in $V$ with pole $y \in V$. 
Then for all $y\in V$ and constants $k, t>0$, we have
 \begin{equation}\label{keyineq1}
        |\{x \in V : |(D^2 u(x))^{-1/2} D_x g_V(x, y)| > t\}| \leq \frac{k}{\lambda t^2} + |\{x \in V : g_{V}(x, y) > k\}|.
        \end{equation}
\end{prop}
\begin{proof} Fix $y\in V$. 
Let \[\xi:= g_V(\cdot, y) \quad \text{and }\phi(x):= (D^2 u(x))^{-1/2} D_x g_V(x, y).\] To simplify, we use this notation
\[\{|\phi|>s\}:=\{x\in V: |\phi(x)|>s\},\quad \{\xi>s\}:=\{x\in V: \xi(x)>s\}.\]
 For constants $t, k >0$, we have
	\begin{equation}\label{intD1}
		 \{|\phi| > t\} \subset \{\xi > k\} \cup \Big(\{|\phi| > t\} \cap \{\xi \leq k\}\Big).
	\end{equation}
   We claim that
	\begin{equation} \label{intD2}
		\int_{\{\xi \leq k\}} U^{ij} D_i \xi D_j \xi \; dx = k.
	\end{equation}
This identity appears in various forms in \cite{Le24} such as equations $(14. 23)$ and $(14.71)$ there. For the reader's convenience, we include its 
proof.
	Indeed, let \[
		\xi_k:=\min\{\xi, k\}.\] 
	Then $\xi_k = \xi$ in $\{\xi < k\}$ and $D\xi_k = 0$ in $\{\xi > k\}$. From Proposition \ref{prop: G_reg}, $\xi_k \in W^{1, 2}_0(V)\cap C(\overline{V})$. Remark \ref{rem: G_prop} tells that we can use $\xi_k$ as a test function in \eqref{Gdef}, so 
	\[
		k = \xi_k(y) = \int_V U^{ij}D_i\xi D_j \xi_k \; dx = \int_{\{\xi \leq k\}} U^{ij}D_i\xi D_j \xi \; dx.
	\]
Thus, \eqref{intD2} is proved as claimed.

\medskip	
	Using \eqref{intD2} and recalling $|\phi|^2 = (D^2 u)^{-1}D\xi \cdot D\xi$, we have
	\begin{equation}\label{intD3}
		\begin{aligned}
			\big| \{|\phi| > t\}  \cap \{\xi \leq k\}\big| \leq \int_{\{\xi \leq k\}} \frac{|\phi|^2}{t^2}\, dx&= \int_{\{\xi \leq k\}} \frac{(D^2 u)^{-1} D \xi \cdot D \xi}{t^2} \; dx\\
			&= t^{-2} \int_{\{\xi \leq k\}} (\det D^2 u)^{-1}U^{ij} D_i\xi D_j \xi\; dx\\
			& \leq \frac{k}{\lambda t^2}.
		\end{aligned}
	\end{equation}
	The proposition follows from \eqref{intD1} and \eqref{intD3}.
\end{proof}

\section{Lorentz space estimates for the Green's function}\label{sec: int_G}

In this section, we prove Theorems \ref{Green_int_thm} and \ref{Green_gl_thm}. 

\medskip
For Theorem \ref{Green_int_thm}, we first establish geometric controls and measure estimates for the superlevel sets of the Green's function $g_{S_u(x_0, h)}(\cdot, y)$ for compactly contained sections $S_u(x_0, h)$. This lemma extends \cite[Theorem 14.11]{Le24} into the weak $L^p$ space at the end point $p = n/(n-2)$ when $n \geq 3$.

\begin{lem}\label{Gvol}
Let $u \in C^3(\Omega)$ be a convex function satisfying \eqref{MAu}, where $\Omega\subset \R^n$.
Assume $S_u\left(x_0, 2 h\right) \Subset \Omega$ where $x_0\in\Omega$ and $h>0$. Let $g_{S_u(x_0, h)}(\cdot, y)$ be the Green's function of the linearized Monge--Amp\`{e}re operator $D_i(U^{ij} D_j)$ in $S_u(x_0, h)$ with pole $y\in S_u(x_0, h)$. There exist constants $\eta(n, \lambda, \Lambda) \in (0, 1)$ and $\tau_0(n, \lambda, \Lambda)>1$ such that the following statements hold.
\begin{enumerate}
\item[(i)] For all $y\in S_u(x_0, h)$, we have
$S_u(y, 2\eta h) \Subset S_u(x_0, 2h)$ and 
for all $t> \tau_0 h^{-\frac{n-2}{2}}$,
\begin{equation}
\label{Gtlarge0}
     \{x \in S_u(x_0, h): g_{S_u(x_0, h)}(x, y) > t\} \subset
     \begin{cases}
         S_u(y, 2\eta h 2^{-t / \tau_0}) &\text{if } n =2,\\
         S_u(y, (4\tau_0)^{\frac{2}{n-2}} t^{-\frac{2}{n-2}}) &\text{if } n \geq 3.
     \end{cases}
\end{equation}
\item[(ii)] For all $y\in S_u(x_0, h)$ and  $t > 0$, we have
\begin{equation}\label{intG_lvl}
    |\{x \in S_u(x_0, h) : g_{S_u(x_0, h)}(x, y) > t\}| \leq 
    \begin{cases}
         C(\lambda, \Lambda)h2^{-t/\tau_0} &\text{ if } n = 2, \\
        C(n, \lambda, \Lambda)t^{-\frac{n}{n-2}}& \text{ if } n\geq 3.
    \end{cases}
\end{equation}
As a consequence, if $n \geq 3$, then
    \[
        \|g_{S_u(x_0, h)}(\cdot, y)\|_{L^{\frac{n}{n-2}, \infty}(S_u(x_0, h))} \leq C(n, \lambda, \Lambda).
    \]
    \end{enumerate}
\end{lem}

\begin{proof} 
The inclusion property of interior sections (Theorem \ref{thm: inclusion}) shows that there is a constant $\eta = \eta(n, \lambda, \Lambda) \in (0, 1)$ such that 
$S_u(y, 2\eta h) \Subset S_u(x_0, 2h)$ for any $y\in S_u(x_0, h)$. Let \[\tilde{g}(x):=  g_{S_u(x_0, 2h)}(x, y) - g_{S_u(x_0, h)}(x, y).\]  As observed at the beginning of the proof of \cite[Lemma 14.7]{Le24}, we have $\tilde{g} \in W^{2, n}_{\text{loc}}(S_u(x_0, h))\cap C(\overline{S_u(x_0, h)})$ and it satisfies 
\[
        U^{ij} D_{ij}\tilde{g} = 0\quad \text{ in } S_u(x_0, h) \quad\text{ and }
        \tilde{g} \geq 0 \quad\text{ on } \partial S_u(x_0, h).
\]
The Aleksandrov--Bakelman--Pucci (ABP) maximum principle (see \cite[Theorem 9.1]{GT}) shows that $\tilde{g} \geq 0$ in $S_u(x_0, h)$. Thus,  \[g_{S_u(x_0, 2h)}(x, y) \geq g_{S_u(x_0, h)}(x, y) \geq 0 \quad\text{for all }x\in S_u(x_0, h).\] Consequently, we have for all $y\in S_u(x_0, h)$ and $t>0$, 
\begin{equation}
\label{compG}
    \{x\in S_u(x_0, h): g_{S_u(x_0, h)}(x, y) > t\} \subset \{x\in S_u(x_0, 2h): g_{S_u(x_0, 2h)}(x, y) > t\}.
\end{equation}

\medskip
To simplify, we denote $S:= S_u(x_0, 2h)$. Fix $y\in S_u(x_0, h)$. By \cite[Lemma 14.10]{Le24}, we have
\[
     \{x \in S: g_S(x, y) > \tau (2h)^{-\frac{n-2}{2}}\} \subset
     \begin{cases}
         S_u(y, 2\eta h 2^{-\tau / \tau_0}) &\text{if } n =2,\\
         S_u(y, 2(4\tau_0\tau^{-1})^{\frac{2}{n-2}}h) &\text{if } n \geq 3,
     \end{cases}
\]
for all $\tau > \tau_0 = C_1\eta^{-n/2}>1$, where $C_1 = C_1(n, \lambda, \Lambda) > 0$ is large. 

It follows that
\begin{equation}
\label{Gtlarge}
     \{x \in S: g_S(x, y) > t\} \subset
     \begin{cases}
         S_u(y, 2\eta h 2^{-t / \tau_0}) &\text{if } n =2,\\
         S_u(y, (4\tau_0)^{\frac{2}{n-2}} t^{-\frac{2}{n-2}}) &\text{if } n \geq 3
     \end{cases}
     \quad \text{if }t> \tau_0 h^{-\frac{n-2}{2}}.
\end{equation}

From \eqref{compG} and \eqref{Gtlarge}, we obtain \eqref{Gtlarge0}. Part (i) is proved.

\medskip
From the volume estimates of sections in Lemma \ref{vol_est}, we obtain for some $C_2(n, \lambda, \Lambda)>0$ and all $t> \tau_0 h^{-\frac{n-2}{2}}$, 
\begin{equation}\label{intG1}
     |\{x \in S: g_S(x, y) > t\}| \leq
     \begin{cases}
         |S_u(y, 2\eta h 2^{- t/ \tau_0})|\leq C_2h 2^{-t/\tau_0} &\text{ if } n =2,\\
         | S_u(y, (4\tau_0)^{\frac{2}{n-2}} t^{-\frac{2}{n-2}})| \leq C_2t^{-\frac{n}{n-2}} &\text{ if } n \geq 3.
     \end{cases}
\end{equation}\par

Consider $ n = 2$. If $0 < t \leq \tau_0$, then  using Lemma \ref{vol_est}, we have 
    \begin{equation}
    \label{Gtsmall2}
         |\{x \in S : g_{S}(x, y) > t\}| \leq |S| \leq C_3(\lambda)h  \leq 2C_3h2^{-t/\tau_0}.
    \end{equation}
    From \eqref{compG}, \eqref{intG1} and \eqref{Gtsmall2},  by choosing $C= \max\{C_2, 2C_3\}$, we obtain \eqref{intG_lvl} for $n=2$. 

Consider now $n \geq 3$.  
 If $0 < t \leq \tau_0 h^{-\frac{n-2}{2}}$,  then  using Lemma \ref{vol_est}, we have 
\begin{equation}
\label{Gtsmall3}
    |\{x \in S: g_S(x, y) > t\}| \leq |S| \leq C_4(n, \lambda)h^{n/2} 
    \leq C_4(n, \lambda, \Lambda)t^{-\frac{n}{n-2}}.
\end{equation}
From \eqref{compG}, \eqref{intG1} and \eqref{Gtsmall3}, by choosing $C= \max\{C_2, C_4\}$, we obtain \eqref{intG_lvl} for 
$n\geq 3$. Part (ii) is proved.
The proof of the lemma is complete.
\end{proof}

\medskip
Now, we are ready to prove Theorem \ref{Green_int_thm}.
\begin{proof}[Proof of Theorem \ref{Green_int_thm}] Denote $S_h:=S_u(x_0, h)$.
Fix $y \in S_h$. Let \[\xi(x):= g_{S_h}(x, y) \quad \text{and }\phi(x):= (D^2 u(x))^{-1/2} D_x g_{S_h}(x, y).\]

\medskip
\noindent
\textbf{Step 1:} We first prove \eqref{int_Dg_level} from which \eqref{int_Dg_1} follows.

 We will use Proposition \ref{keyprop} so it remains to estimate the measure $|\{x\in S_h: \xi(x) > k\}|$ from above and then optimize over $k$. 
 
        Combining Proposition \ref{keyprop} and Lemma \ref{Gvol}, we obtain for all $k>0$ and $t>0$
    \begin{equation}\label{intD5}
         \big| \{x\in S_h: |\phi(x)| > t\} \big| \leq 
         \begin{cases}
            \frac{k}{\lambda t^2} + C_1(\lambda, \Lambda)h2^{-k/\tau_0} &\text{ if } n = 2, \\
             \frac{k}{\lambda t^2} + C_1(n, \lambda, \Lambda)k^{-\frac{n}{n-2}} &\text{ if } n \geq 3,
         \end{cases}
    \end{equation}
    where $\tau_0(\lambda, \Lambda) > 1$. 
    
    Letting in \eqref{intD5}
    \begin{equation*}
     k=     \begin{cases}
     h+ 2 \log_{2^{1/\tau_0}} \max\{t, 1\} &\text{ if }n = 2,\\  t^{\frac{n-2}{n-1}} &\text{ if }n \geq 3,
             \end{cases}
    \end{equation*}
        we obtain \eqref{int_Dg_level}.
        
    \medskip
    \noindent
	\textbf{Step 2:} We prove \eqref{int_Dg_2} for $p \in \big( 1, \frac{n}{n-1}\big)$.

	Consider  $n \geq 3$.
    Let $\chi_{S_h}$ be the characteristic function of $S_h$.  
    By the H\"{o}lder inequality in Lorentz spaces (Theorem \ref{thm: Hol_Lor}) and \eqref{int_Dg_1}, we have
    \begin{equation}\label{lor_1}
    	\begin{aligned}
        	\int_{S_h} |\phi|^p \; dx = \|\phi\|_{L^{p, p}(S_h)}^p &\leq C_2(n, p)\|\phi\|_{L^{\frac{n}{n-1}, \infty}(S_h)}^p\|\chi_{S_h}\|_{L^{\frac{np}{n+p-np}, p}(S_h)}^p\\
        	& \leq C_3(n, \lambda, \Lambda, p) \int_0^\infty t^{p-1}|\{x\in S_h: |\chi_{S_h}(x)| > t\}|^{\frac{n+p-np}{n}}\; dt\\
        	& = C_3 \int_0^1 t^{p-1}|S_h|^\frac{n+p-np}{n}\; dt\\
        	& \leq C_4(n, \lambda, \Lambda, p) h^{\frac{n}{2}- \frac{n-1}{2}p},
    	\end{aligned}
    \end{equation}
    where the last inequality comes from Lemma \ref{vol_est}. This proves \eqref{int_Dg_2} for $n \geq 3$.

    \medskip
    Consider  now $ n = 2$. Then $p\in (1, 2)$. 
    From \eqref{intD5}, and for $k > 0$ and $q \in (1, \infty)$ to be chosen, we have
    \[
         | \{x\in S_h: |\phi (x)| > t \}| \leq \frac{k}{\lambda t^2} + C_1(\lambda, \Lambda)2^{-k/\tau_0}h \leq \frac{k}{\lambda t^2} + C_5(\lambda, \Lambda, q)k^{-q}h.
    \]
     Let $k = t^\frac{2}{1+q}h^{\frac{1}{1+q}}$ so that $kt^{-2} = k^{-q}h$.  Then
       \[
         | \{x\in S_h: |\phi (x)| > t \}| \leq C_6(\lambda,\Lambda, q) h^{\frac{1}{1+q}} t^{\frac{-2q}{1+q}}.
         \]
      By the layer-cake formula (Lemma \ref{lcf}), we have for every $\tau > 0$
	\[
		\begin{aligned}
			\int_{S_{h}} |\phi|^p \; dx &= \int_{ \{x\in S_h: |\phi(x)| \leq \tau\}} |\phi|^p \; dx + \int_{\{x\in S_h: |\phi(x)| > \tau\}} |\phi|^p \; dx\\
			& \leq \tau^p|{S_{h}}| + p\int_{\tau}^\infty t^{p-1}\big| \{x\in S_h: |\phi(x)| > t\} \big| \; dt\\
			& \leq C_7(\lambda, \Lambda)\tau^p h  + C_8(\lambda, \Lambda, p, q)h^{\frac{1}{1+q}}\int_{\tau}^\infty t^{p-1-\frac{2q}{1+q}} \; dt.
            \end{aligned}
        \]
        We may choose $q = \frac{p+1}{2-p}$ so that $p - \frac{2q}{1+q} < 0$ (due to $p\in (1, 2)$) and the integral above converges. Then,
        \[
            \int_{S_{h}} |\phi|^p \; dx \leq C_7\tau^p h + C_9(\lambda, \Lambda, p)h^{\frac{1}{1+q}}\tau^{p - \frac{2q}{1+q}}.
        \]
        Setting $\tau = h^{-1/2}$, we obtain
        \eqref{int_Dg_2} for $n = 2$. This completes the proof.
\end{proof}

\medskip

For Theorem \ref{Green_gl_thm}, extending \cite[Theorem 14.22]{Le24} into the weak $L^p$ space at the end point $p = \frac{n}{n-2}$, we establish the following geometric controls on the superlevel sets and distribution function estimates for the Green's function $g_V$. Here, $V$ is no longer required to be compactly contained in the domain but we need suitable global conditions on the domain and the Monge--Amp\`ere potential. 

\begin{lem} \label{lem: glpG} Let $u$ and $\Omega$ satisfy \eqref{glb_1}--\eqref{glb_4}. 
Let $V\subset\Omega$ be open, proper, and contained in a section $S_u(x_0, h)$ of $u$ of height $h>0$. Let $g_V(\cdot, y)$ be the Green's function of the linearized Monge--Amp\`{e}re operator $D_i(U^{ij} D_j)$ in $V$ with pole $y \in V$.
\begin{enumerate}
\item[(i)]  There exist positive constants $C_\ast(n, \lambda, \Lambda, \rho)$,  $\bar C(\lambda,\Lambda,\rho)$ such that for all $y\in V$, we have
\begin{equation}
\label{gbGtlarge}
     \{x \in V: g_V(x, y) > t\} \subset
     \begin{cases}
         S_u(y, h 2^{-t /\bar{C}}) &\text{if } n =2,\\
         S_u(y, C_\ast t^{-\frac{2}{n-2}}) &\text{if } n \geq 3
     \end{cases}
     \quad \text{if }t> C_\ast(n, \lambda,\Lambda,\rho).
\end{equation}
\item[(ii)] For all $y \in V$ and all $t>0$, we have
\begin{equation}\label{glbG_lvl}
    |\{x \in V : g_{V}(x, y) > t\}| \leq 
    \begin{cases}
        C(\lambda, \Lambda, \rho) h2^{-t/C_1} &\text{ if } n = 2, \\
        C(n, \lambda, \Lambda, \rho)t^{-\frac{n}{n-2}} &\text{ if } n \geq 3,
    \end{cases}
\end{equation}
where $C_1 = C_1(\lambda, \Lambda, \rho) > 1$. Consequently, 
    \begin{equation*}
    	\|g_V(\cdot, y)\|_{L^{\frac{n}{n-2}, \infty}(V)} \leq C(n, \lambda, \Lambda, \rho) \quad\text{if } n\geq 3.
    \end{equation*}
    \end{enumerate}
\end{lem}

\begin{proof}
     From the bound on the Green's function \cite[Theorem 14.21]{Le24},  we have 
     $$\max_{x\in \partial S_u (y, t_0)}g_\Omega(x, y)\leq \begin{cases} 
     C_1(n,\lambda,\Lambda, \rho)\log_2 (1/t_0) & \mbox{if } n = 2,\\
     C_1(n,\lambda,\Lambda, \rho)t_0^{-\frac{n-2}{2}} &\mbox{if } n \geq 3, 
    \end{cases} 
$$
for all $y\in\Omega$ and $0<t_0< c=c(n, \lambda, \Lambda, \rho) < 1$.  Therefore, for $0 < s < c$, we have
    \begin{equation}\label{glbG1}
    S_u(y, s) \supset
    \begin{cases}
        \{x \in \Omega: g_\Omega(x, y) > C_1(\lambda, \Lambda, \rho) \log_2 (1/s)\} &\text{ if } n = 2, \\
        \{x \in \Omega: g_\Omega(x, y) > C_1(n, \lambda, \Lambda, \rho)s^{-\frac{n-2}{2}}\} &\text{ if } n \geq 3.
    \end{cases}
    \end{equation}

        We first prove \eqref{glbG_lvl} for the case $n\geq 3$. Fix $y\in V$. 
    As in \eqref{compG}, we have
    \begin{equation*}
    	\{x \in V: g_V(x, y) > t\} \subset \{x \in \Omega: g_\Omega(x, y) > t\} \quad \text{ for all } t > 0.
    \end{equation*}
    From \eqref{glbG1}, we see that for $t > C_1c^{-\frac{n-2}{2}}$, 
    \begin{equation}
    \label{gbGt3}
    \{x \in V: g_V(x, y) > t\} \subset \{x \in \Omega: g_\Omega(x, y) > t\}\subset S_u\big(y, \big(t/C_1\big)^{-\frac{2}{n-2}}\big).
    \end{equation}
    By Lemma \ref{vol_est}, we have for $t > C_1c^{-\frac{n-2}{2}}$, 
    \begin{equation}\label{glbG4}
             |\{x\in V: g_V(x, y) > t \}| \leq \big|S_u\big(y, \big(t/C_1\big)^{-\frac{2}{n-2}}\big)\big| \leq C_2(n, \lambda, \Lambda, \rho)t^{-\frac{n}{n-2}}.
    \end{equation}
    For $0 < t \leq C_1c^{-\frac{n-2}{2}}$, recalling $\Omega\subset B_{1/\rho}(0)\subset\R^n$,  we have
        \begin{equation}\label{glbG5}
        |\{x\in V: g_V(x, y) > t \}| \leq |\Omega| \leq C_3(n, \lambda, \Lambda, \rho) t^{-\frac{n}{n-2}}.
    \end{equation}
    Combining \eqref{glbG4} and \eqref{glbG5}, we obtain  \eqref{glbG_lvl} for all $t > 0$ when $n\geq 3$.

    \medskip
       We prove \eqref{glbG_lvl} for the case
    $n = 2$. As remarked after the statement of  Theorem \ref{Green_gl_thm}, $\Omega$ is itself a section of $u$ with height comparable to $1$.  Fix $y\in V$.  Note that 
     \begin{equation*}
    	\{x \in V: g_V(x, y) > t\} \subset \{x \in S_u(x_0, h)\cap \Omega: g_{S_u(x_0, h)\cap \Omega}(x, y) > t\} \quad\text{ for all } t > 0.
    \end{equation*}
     By repeating the proof of \cite[Theorem 14.21]{Le24} for $g_{S_u(x_0, h)\cap \Omega}(x, y)$ instead of $g_\Omega(x, y)$ (for example, in \cite[(14.48)]{Le24}, we can replace $c$ by $h$ and $c_2$ by $c_2h$), we obtain 
     $$\max_{x\in \partial S_u (y, t_0)}g_{S_u(x_0, h)\cap \Omega}(x, y)\leq 
C_4(\lambda,\Lambda, \rho) \log_2 (h/t_0)  $$
for all $y\in S_u(x_0, h)\cap \Omega$ and $0<t_0< ch$.  
     
     \medskip
It follows that for all $y\in V$ and all $t > C_4 \log_2 (1/c)$, 
    \begin{equation}
    \label{gbGt2}
         \{x \in V: g_V(x, y) > t\}\subset \{x \in S_u(x_0, h)\cap \Omega: g_{S_u(x_0, h)\cap \Omega}(x, y) > t\} \subset S_u(y, h 2^{-t/C_4}).
    \end{equation}
    From \eqref{gbGt3} and \eqref{gbGt2}, we obtain \eqref{gbGtlarge}, which is a global version of \eqref{Gtlarge0}. Part (i) is proved.
    
    \medskip
  
        Using the volume estimates for sections (Lemma \ref{vol_est}), we obtain for $t > C_4 \log_2 (1/c)$, 
\begin{equation}\label{glbG2}
        |\{x \in V: g_V(x, y) > t\}| \leq |S_u(y, h2^{-t/C_4})|\leq C(\lambda)h2^{-t/C_4}.
    \end{equation}
    On the other hand, if $ 0 < t \leq C_4 \log_2 (1/c)$, then
    \begin{equation}\label{glbG3}
           |\{x \in V: g_V(x, y) > t\}| \leq |S_u(x_0, h)| \leq C(\lambda)h  \leq C_5(\lambda, \Lambda, \rho)h 2^{-t/C_4}.
    \end{equation}
    Combining \eqref{glbG2} and \eqref{glbG3}, we obtain
     \eqref{glbG_lvl} for all $ t > 0$ when $n=2$. Part (ii) is completely proved and so is
the lemma.
\end{proof}

\begin{proof}[Proof of Theorem \ref{Green_gl_thm}]
	The proof is similar to that of Theorem \ref{Green_int_thm}. It uses Proposition \ref{keyprop} and Lemma  \ref{lem: glpG} instead of Lemma \ref{Gvol}, so we skip it.
	\end{proof}

\section{Maximum principles for linearized Monge--Amp\`ere equations} 
\label{sec_max}
In this section and the next, we present some applications of Theorems \ref{Green_int_thm} and \ref{Green_gl_thm} to the regularity of solutions to linearized Monge--Amp\`ere equations in divergence form.
This section focuses on maximum principles for subsolutions to \eqref{eqdiv}. They imply
uniform estimates 
for solutions to 
\eqref{eqdiv} with zero boundary value.

\medskip
We begin with the case of compactly contained sections.

\begin{prop}\label{intmax}
      Let $\Omega \subset \R^n$  be a convex domain and $u\in C^3(\Omega)$ be a convex function satisfying \eqref{MAu}. 
     Assume
      $\F \in  W^{1, n}_{\text{loc}}(\Omega; \R^n)$, $\mu\in L^n_{\text{loc}}(\Omega)$, and $S_u(x_0, 2h) \Subset \Omega$ where $x_0\in\Omega$ and $h>0$. Assume $v\in W^{2, n}_{\text{loc}}(S_u(x_0, h))\cap C(\overline{S_u(x_0, h)})$ satisfies
    \[        - U^{ij}D_{ij}v \leq \diver \F +\mu \; \quad \text{ in } \quad S_u(x_0, h).
 \]
    \begin{enumerate}
        \item[(i)] Assume $n \geq 3$, $\mu=0$, and $ (D^2u)^{1/2} \F \in L^{n, 1}(S_u(x_0, h);\R^n)$. Then 
        \[
            \sup_{S_u(x_0, h)} v \leq \sup_{\partial S_u(x_0, h)} v + C(n, \lambda, \Lambda)\| (D^2u)^{1/2} \F\|_{L^{n, 1}(S_u(x_0, h))}.
        \]
        \item[(ii)] Assume 
        $(D^2u)^{1/2} \F \in L^{q}(S_u(x_0, h);\R^n)$ for some $q > n$ and 
        there exist $M_0\geq 0$ and  $\e>0$ such that for $\mu^{+}=\max\{\mu, 0\}$, we have
  \begin{equation}
  \label{mump}
  \mu^{+}(S_u(z, s)) \leq M_0 s^{\frac{n-2}{2} +\e}\quad \text{for all sections } S_u(z, s)\Subset S_u(x_0, 2h).\end{equation}
         Then
        \[
             \sup_{S_u(x_0, h)} v \leq \sup_{\partial S_u(x_0, h)} v + C_\ast(n, \lambda, \Lambda, q)\| (D^2u)^{1/2} \F\|_{L^q(S_u(x_0, h))} h^{\frac{q-n}{2q}} + C_\star(n, \lambda, \Lambda, \e)M_0 h^{\e}.
        \]
        \end{enumerate}
\end{prop}

\begin{proof}
    Let us denote $S_h:= S_u(x_0, h)$. 
    Let $\psi\in W^{2, n}_{\text{loc}}(S_h)\cap W^{1,2}_0(S_h)\cap C(\overline{S_h})$ solve
	\[
		-U^{ij}D_{ij}\psi = \diver \F +\mu \quad \text{ in } S_h \; \quad \text{and }\quad \psi = 0 \quad \text{ on } \partial S_h;
	\]
	see \cite[Theorem 9. 30]{GT}.
	From the assumption, we have
	\[
		-U^{ij} D_{ij}(v - \psi) \leq 0\quad  \text{ in } S_h.
	\]
	Since $v - \psi \in W^{2, n}_{\text{loc}}(S_h)\cap C(\overline{S_h})$, the ABP maximum principle gives 
	\[
		\sup_{S_h} v - \sup_{S_h} \psi \leq \sup_{S_h} (v - \psi) \leq \sup_{\partial S_h} (v - \psi) = \sup_{\partial S_h} v.
	\]
	Hence, 
	\begin{equation} \label{gmp1}
		\sup_{S_h} v \leq \sup_{\partial S_h} v + \sup_{S_h} \psi.
	\end{equation}
    
    For $y \in S_h$, let $ g_{S_h}(x, y)$ be the Green's function of the linearized Monge--Amp\`{e}re operator $D_i(U^{ij} D_j)$ in $S_h$ with pole $y$. 
     Using the representation formula \eqref{G_rep} and integration by parts, we find 
	\begin{equation}
			\label{int_ubd1}
	\begin{split}
\psi(y)  = \int_{S_h} g_{S_h}(\cdot, y) (\diver \F+\mu)\, dx 
&=-  \int_{S_h} D_x g_{S_h}(x, y) \cdot \F(x) \,  dx +  \int_{S_h} g_{S_h}(\cdot, y) \mu\, dx \\
			& \leq  \int_{S_h} |(D^2 u)^{-1/2} D_x g_{S_h}(x, y) \cdot (D^2 u)^{1/2} \F|\,  dx\\
			&\quad\quad + \int_{S_h} g_{S_h}(\cdot, y) \mu\, dx:= \psi_1(y) + \psi_2(y).
\end{split}
	\end{equation}

We prove part (i) when $n\geq 3$ and $\mu=0$. 	By Theorem \ref{Green_int_thm} and the H\"{o}lder inequality in Lorentz spaces (Theorem \ref{thm: Hol_Lor}), we have
    \begin{equation}
    \label{int_ubd13}
    \begin{split}
        \psi(y) \leq \psi_1(y)&\leq C_0(n)\|(D^2 u)^{-1/2} D_x g_{S_h}(\cdot, y)\|_{L^{\frac{n}{n-1}, \infty}(S_h)} \|(D^2u)^{1/2} \F\|_{L^{n, 1}(S_h)}\\  &\leq C_1(n, \lambda, \Lambda)\|(D^2u)^{1/2} \F\|_{L^{n, 1}(S_h)}.
        \end{split}
    \end{equation}
    Since $y \in S_h$ is arbitrary, \eqref{gmp1}--\eqref{int_ubd13} establish part (i).
    
    \medskip
  We prove part (ii) by estimating $\psi_1(y)$ and $\psi_2(y)$. 
  Recall from \eqref{int_ubd1} that
   \[\psi_1(y)  = \int_{S_h} |(D^2 u)^{-1/2} D_x g_{S_h}(x, y) \cdot (D^2 u)^{1/2} \F|\,  dx.\]
   Since $q>n$, we have
   $\frac{q}{q-1} < \frac{n}{n-1}$. 
  Applying the H\"{o}lder inequality and Theorem \ref{Green_int_thm}, we have
    \begin{equation}
    \label{int_ubd23}
        \begin{split}
        \psi_1(y) &\leq \|(D^2u)^{-1/2}D_x g_{S_h}(\cdot, y)\|_{L^{\frac{q}{q-1}}(S_h)} \|(D^2u)^{1/2}  \F\|_{L^q(S_h)} \\ &\leq C_2(n, \lambda, \Lambda, q)\|(D^2u)^{1/2}  \F\|_{L^q(S_h)}h^{\frac{q-n}{2q}}.
                \end{split}
    \end{equation}

       \medskip 
       We estimate \[\psi_2(y)=\int_{S_h} g_{S_h}(\cdot, y) \mu\, dx.\] By Lemma \ref{Gvol}, there are constants $\eta(n, \lambda, \Lambda) \in (0, 1)$ and $\tau_0(n, \lambda, \Lambda)>1$ such that for any $y\in S_h$, we have
$S_u(y, 2\eta h) \Subset S_u(x_0, 2h)$   and 
\begin{equation}
\label{Gtlarge2}
     \{x \in S_h: g_{S_h}(x, y) > t\} \subset
     \begin{cases}
         S_u(y, 2\eta h 2^{-t / \tau_0}) &\text{if } n =2,\\
         S_u(y, (4\tau_0)^{\frac{2}{n-2}} t^{-\frac{2}{n-2}}) &\text{if } n \geq 3
     \end{cases}
     \quad \text{if }t> \tau_0 h^{-\frac{n-2}{2}}.
\end{equation}
Thus, by choosing $\tau_1= \tau_1(n, \lambda, \Lambda)$ large enough, we see that all sections in \eqref{Gtlarge2} are contained in $S_u(y, \eta h)\Subset S_u(x_0, 2h)$ when $t>\tau_1 h^{-\frac{n-2}{2}}=: T$. 

Consider $n\geq 3$.
Using the layer-cake formula (Lemma \ref{lcf}) and recalling \eqref{mump} and that $g_{S_h}(x, y)\geq 0$, we can estimate
        \begin{equation}
        \label{int_mu3}
           \begin{split}
           \psi_2(y)  
           &\leq \int_{S_h} g_{S_h}(x, y) \mu^{+}(x)\, dx \\&=\int_0^{\infty} \mu^+\big(\{x \in S_h: g_{S_h}(x, y)>t\}\big)\, dt\\
           &=\int_0^T \mu^{+}\big(\{x \in S_h: g_{S_h}(x, y)>t\}\big)\, dt + \int_T^{\infty} \mu^{+}\big(\{x \in S_h: g_{S_h}(x, y)>t\}\big)\, dt\\
           &\leq T \mu^{+}(S_h) +  \int_T^{\infty} \mu^{+} \big(S_u(y, (4\tau_0)^{\frac{2}{n-2}} t^{-\frac{2}{n-2}})\big)\, dt\\
           &\leq M_0T h^{\frac{n-2}{2} +\e} + C_3 M_0\int_T^{\infty} t^{-\frac{2}{n-2}(\frac{n-2}{2} +\e)}\, dt\\
           &\leq M_0T h^{\frac{n-2}{2} +\e} + C_4 M_0 T^{-\frac{2\e}{n-2}} \\
           &\leq C_5(n,\lambda,\Lambda,\e)M_0 h^{\e}.
                                \end{split}
    \end{equation}
    
    Consider $n=2$. Then $T=\tau_1$, and as above, we have 
          \begin{equation}
        \label{int_mu2}
           \begin{split}
           \psi_2(y)  &\leq T \mu^{+}(S_h) +  \int_T^{\infty} \mu^{+} \big(S_u(y, 2\eta h 2^{-t/\tau_0})\big)\, dt\\
           &\leq M_0T h^{\e} + C_6 h^{\e} M_0\int_T^{\infty} 2^{-t\e/\tau_0}\, dt\\
           &\leq C_7(\lambda,\Lambda,\e)M_0 h^{\e}.
                                        \end{split}
    \end{equation}
    Combining \eqref{int_ubd1} with \eqref{int_ubd23}, \eqref{int_mu3}, and  \eqref{int_mu2}, we obtain the asserted estimate in part (ii).
  The proposition is proved.
\end{proof}
\begin{rem}  For a nonnegative Radon measure $\mu$, we define the truncated Riesz potential $I_u^{\mu}$ with respect to $u$ by
\begin{equation}
I_u^\mu(y, h) =\int_0^h \frac{\mu(S_u(y, s))}{s^{\frac{n}{2}}} ds.
\end{equation}
Then, the estimates for $\psi_2(y)$ in the proof of Proposition \ref{intmax} (ii) can be expressed using $I_u^{\mu^+}$. Indeed, by the change of variables 
\begin{equation*}
s= \begin{cases} (4\tau_0)^{\frac{2}{n-2}} t^{-\frac{2}{n-2}} &\quad\text{when } n\geq 3,  \\ 
 2\eta h 2^{-t/\tau_0}& \quad \text{when } n=2, 
\end{cases}
\end{equation*}
we have in \eqref{int_mu3} and \eqref{int_mu2}
\begin{equation}
\psi_2(y) \leq Ch^{-\frac{n-2}{2}}  \mu^{+}(S_u(x_0, h)) + C I_u^{\mu^+} (y, \eta h),
\end{equation} 
where $C= C(n,\lambda,\Lambda)>0$ and $\eta=\eta(n,\lambda,\Lambda)\in (0, 1)$ is such that $S_u(y, 2\eta h)\Subset S_u(x_0, 2h)$.
\end{rem}
\medskip
Next, we state a boundary version of Proposition \ref{intmax}. \begin{prop}\label{thm: glb_mp}
    Assume $u$ and $\Omega$ satisfy \eqref{glb_1}--\eqref{glb_4}. Let 
     $V\subset\Omega$ be open, proper, and contained in a section $S_u(x_0, h)$ of $u$ of height $h>0$. 
    Assume $\F \in W^{1, n}(V; \R^n)$, $\mu\in L^n(\Omega)$, and $v \in W^{2, n}_{\text{loc}}(V) \cap C(\overline{V})$ satisfies
	\[
	-U^{ij}D_{ij} v \leq \diver \F +\mu  \quad \text{ in } V.
	\]
    \begin{enumerate}
        \item[(i)] Assume $n \geq 3$, $\mu=0$, and
        $ (D^2u)^{1/2} \F \in L^{n, 1}(V;\R^n)$. Then
        \[
            \sup_V v \leq \sup_{\partial V} v + C(n, \lambda, \Lambda, \rho) \big\| (D^2u)^{1/2} \F\big\|_{L^{n, 1}(V)}.
        \]
        \item[(ii)] Assume 
        $(D^2u)^{1/2} \F \in L^{q}(V;\R^n)$ for some $q > n$ and 
        there exist
        $M_0\geq 0$ and  $\e>0$ such that for $\mu^{+}=\max\{\mu, 0\}$, we have
          \[\mu^+(S_u(z, s)) \leq M_0 s^{\frac{n-2}{2} +\e}\quad \text{for all sections } S_u(z, s)\subset \overline{\Omega}.\]
            Then,    
            \[
            	\sup_V v \leq \sup_{\partial V} v + C_\ast(n, \lambda, \Lambda,q, \rho) \big\| (D^2u)^{1/2} \F\big\|_{L^q(V)} h^{\frac{q-n}{2q}} + C_\star(n, \lambda, \Lambda, \rho, \e)M_0 h^{\e}.
            \]
     \end{enumerate}
\end{prop}

\begin{proof} Our proof is similar to that of Proposition \ref{intmax}. It uses the representation formula \eqref{G_rep}, 
 Lemma \ref{lem: glpG} instead of Lemma \ref{Gvol}
for part (ii), and Theorem \ref{Green_gl_thm} instead of Theorem \ref{Green_int_thm}, so we skip it.
	\end{proof}
As discussed after the statement of Theorem \ref{Green_gl_thm}, the above maximum principle is applicable when $V=\Omega$ where the height is bounded by a constant $M(n, \lambda, \Lambda,\rho)$.

\section{Interior regularity for linearized Monge--Amp\`ere equations}
\label{sec_Har}
In this section, we prove an interior Harnack inequality and interior H\"older estimates for solutions to \eqref{eqdiv}.
\subsection{Interior Harnack inequality} In this subsection, we prove Theorem \ref{int_Har}. The main tools are
 uniform bounds for solutions to \eqref{eqdiv} with zero boundary value
 and the Caffarelli--Guti\'{e}rrez Harnack inequality 
for the linearized Monge--Amp\`{e}re equation \cite[Theorem 5]{CG97}.

\begin{thm}
[Caffarelli--Guti\'{e}rrez Harnack inequality] 
\label{thm: CG_Harnack}
Let $u \in C^2(\Omega)$ be a convex function satisfying \eqref{MAu} in a domain $\Omega$ in $\mathbb{R}^n$. Let $v \in W_{\text {loc }}^{2, n}(\Omega)$ be a nonnegative solution of the linearized Monge--Amp\`{e}re equation
$U^{i j} D_{i j} v=0$ in a section $S_u\left(x_0, 2 h\right) \Subset \Omega$.
Then
\[
\sup _{S_u\left(x_0, h\right)} v \leq C(n, \lambda, \Lambda) \inf _{S_u\left(x_0, h\right)} v.
\]
\end{thm}

We are now ready to prove Theorem \ref{int_Har}.
\begin{proof}[Proof of Theorem \ref{int_Har}]
    Denote $S_h:= S_u(x_0, h)$ and $S_{h/2}:= S_u(x_0, h/2)$. Let $w \in W^{2, n}_{\text{loc}}(S_{h}) \cap C(\overline{S_h})$ satisfy
    \[
            U^{ij} D_{ij} w = \diver \F  +\mu \; \quad\text{ in } S_h \quad \text{ and }\quad
            w = 0 \; \quad \text{ on } \partial S_h;
    \]
    see \cite[Theorem 9.30]{GT}.
    Notice that $v - w \in W^{2, n}_{\text{loc}}(S_h) \cap C(\overline{S_h})$ satisfies
    \[
            U^{ij} D_{ij} (v - w) = 0 \; \quad \text{ in } S_h\quad\text{ and }
            v - w = v \geq 0 \; \text{ on } \partial S_h.
    \]
    Hence, the ABP maximum principle shows that $v - w \geq 0$ in $\overline{S_h}$. Then, applying the Caffarelli--Guti\'{e}rrez Harnack inequality (Theorem \ref{thm: CG_Harnack}) to $v - w$, we find
    \[
        \sup_{S_{h/2}} (v - w) \leq C_1(n, \lambda, \Lambda) \inf_{S_{h/2}} (v - w).
    \]
    Thus 
    \begin{equation}\label{intH1}
            \sup_{S_{h/2}} v 
                 \leq (1+ C_1)\big( \inf_{S_{h/2}} v +      
                \| w\|_{L^\infty(S_h)}\big).
    \end{equation}
    From Proposition \ref{intmax} applied to $w$ and $-w$, we have
    \[
        \|w\|_{L^\infty(S_h)} \leq 
        \begin{cases}
        	C_2(n,\lambda,\Lambda)\| (D^2u)^{1/2} \F\|_{L^{n, 1}(S_h)} \quad &\text{ in part (i)},\\
        	C_3\| (D^2u)^{1/2} \F\|_{L^{q}(S_h)}h^{\frac{q-n}{2q}} + C_4 M_0h^{\e}\quad &\text{ in part (ii)}, 
        \end{cases}
    \]
    where $C_2 = C_2(n, \lambda, \Lambda)> 0$, $C_3 = C_3(n, \lambda, \Lambda, q)> 0$, and $C_4 = C_4(n, \lambda, \Lambda, \e)> 0$.
    Combining this with \eqref{intH1} completes the proof of the theorem.
\end{proof}

\subsection{Interior H\"{o}lder estimates} In this subsection, using Theorem \ref{int_Har}, we give a new proof of the following theorem, first obtained by Wang \cite[Theorem 1.5]{Wang25} and Cui--Wang--Zhou \cite[Theorem 1.2]{CWZ},  on 
 interior H\"{o}lder estimates for \eqref{eqdiv}.

\begin{thm}
\label{thm: int}
 	Let $\Omega \subset \R^n$ be a convex domain and $u\in C^3(\Omega)$ be a convex function satisfying \eqref{MAu}. Let
	$S_u(x_0, 2h_0) \Subset \Omega$ where $x_0\in\Omega$ and $h_0>0$ and $v\in W^{2, n}_{\text{loc}}(S_u(x_0, 2h_0))\cap C(\overline{S_u(x_0, 2h_0)})$ be a solution to \[U^{ij}D_{ij}v = \diver \F +\mu  \quad\text{in}\quad S_u(x_0, 2h_0),\] 
	where $\F \in  W^{1, n}_{\text{loc}}(\Omega; \R^n)$ with
        $(D^2u)^{1/2} \F \in L^{q}(S_u(x_0, 2h_0);\R^n)$ for some $q > n$, and $\mu\in L^{n}(S_u(x_0, 2h_0))$ and 
        there exist $M_0\geq 0$ and  $\e>0$ such that 
  \[|\mu|(S_u(z, s)) \leq M_0 s^{\frac{n-2}{2} +\e}\quad \text{for all sections } S_u(z, s)\Subset S_u(x_0, 2h_0).\]
  Letting
  \[ M:= \sup_{x, y\in S_u(x_0, h_0)} |Du(x) - Du(y)|\quad\text{and}\quad \tilde M:=h_0^{\frac{q-n}{2q}}\|(D^2u)^{1/2}  \F\|_{L^q(S_u(x_0, h_0))} + M_0 h_0^{\e}. \]
  Then
\begin{equation}\label{int_est_qmu}
	|v(y) - v(z)|  \leq C\big( \|v\|_{L^\infty(S_u(x_0, h_0))} + \tilde M \big)h_0^{-\gamma}M^{\gamma}|y - z|^{\gamma}\quad\quad\text{for all } y, z \in S_u(x_0, h_0/2),
\end{equation}
where $C = C(n, \lambda, \Lambda, q, \e) > 0$ and $\gamma=\gamma(n, \lambda, \Lambda, q, \e)\in (0,1)$.
\end{thm}

\begin{rem} Some remarks on Theorem \ref{thm: int} are in order.
\begin{enumerate}
\item If $\F=0$ in Theorem  \ref{thm: int}, the term $ \|v\|_{L^\infty(S_u(x_0, h_0))}$ in \eqref{int_est_qmu} can be replaced by $ \|v\|_{L^p(S_u(x_0, h_0))}$ for any $p>0$ in \cite{CWZ}.
\item From Caffarelli's $C^{1, \alpha}$ estimates \cite{Caff91} (see also \cite[Corollary 5.22]{Le24}), we have
	\begin{equation}
	\label{Mbound}
		M:= \sup_{x, y\in S_u(x_0, h_0)} |Du(x) - Du(y)| \leq C_1(n, \lambda, \Lambda, \diam(S_u(x_0, 2h_0))) h_0^{-\alpha}, 
	\end{equation}
    where $\alpha= \alpha(n, \lambda, \Lambda) > 0$. 
\end {enumerate}
\end{rem}
\begin{proof}[Proof of Theorem \ref{thm: int}] We divide the proof into two steps.

\medskip
\noindent
{\bf Step 1:} Oscillation decay of solutions to \eqref{eqdiv} in small sections. We claim that there exist $C > 0$ and $\gamma \in (0, 1)$ depending only on $n, \lambda, \Lambda, q, \e$ such that for any $0 < h \leq h_0$, we have
    \begin{equation}
    \label{osc_est}
        \underset{S_u(x_0, h)}{\osc} v \leq C\big(\frac{h}{h_0}\big)^\gamma\big( \underset{S_u(x_0, h_0)}{\osc} v + h_0^{\frac{q-n}{2q}}\|(D^2u)^{1/2}  \F\|_{L^q(S_u(x_0, h_0))} + M_0 h_0^{\e} \big).
    \end{equation}

    Denote $S_h:= S_u(x_0, h)$.  Since $\bar v:= v- \inf_{S_h} v$ is a nonnegative solution to \[U^{ij} D_{ij}\bar v=  \diver \F + \mu \quad\text{in}\quad S_h,\] we can apply Theorem \ref{int_Har}
    to obtain for some $C(n,\lambda,\Lambda)>3$
    \[C^{-1}\sup_{S_{h/2}} \bar v \leq \inf_{S_{h/2}} \bar v +C_1(n, \lambda, \Lambda, q)\big\|(D^2u)^{1/2} \F\big\|_{L^q(S_h)}h^{\frac{q-n}{2q}} + C_2(n, \lambda, \Lambda, \e)M_0h^{\e}. \]
    Thus, 
	\begin{equation}\label{ineq: int_osc_2}
	\begin{split}
	    \underset{S_{h/2}}{\osc}  v = \sup_{S_{h/2}} \bar v  - \inf_{S_{h/2}} \bar v  &\leq (1 - C^{-1}) \sup_{S_{h/2}} \bar v  + C_1h^{\frac{q-n}{2q}}\big\|(D^2u)^{1/2} \F\big\|_{L^q(S_h)} + C_2 M_0h^{\e}\\
	    &\leq \beta \underset{S_h}{\osc}  v + C_1h^{\frac{q-n}{2q}}\big\|(D^2u)^{1/2} \F\big\|_{L^q(S_h)}h^{\frac{q-n}{2q}} + C_2M_0h^{\e},
	    \end{split}
	\end{equation}
	where we can choose \[\beta= \beta(n, \lambda, \Lambda, q,\e)\in (1-C^{-1}, 1) \quad\text{such that}\quad  \min\big\{\beta 2^{\frac{q-n}{2q}}, \beta 2^{\e}\big\}>1.\]
       A standard iteration argument (see, for example, the proof of \cite[Theorem 12.13]{Le24})
   gives \eqref{osc_est} with $\gamma =\log(1/\beta)\log 2\in (0, 1)$ as claimed.

\medskip
\noindent
{\bf Step 2:} Interior H\"{o}lder estimates. 
	Let $y, z \in S_{h_0/2}$. By the inclusion property of sections (Theorem \ref{thm: inclusion}), there exists $\tau = \tau(n, \lambda, \Lambda)>0$, such that $S_u(y, 2\tau h_0) \subset S_{h_0}\Subset \Omega$. Then, 
	we obtain the same estimates as \eqref{osc_est} in $S_u(y, h)$ for $h \leq \tau h_0$. We consider two cases: 

\medskip	
	\textbf{Case 1:}  $z \in S_u(y, \tau h_0)$. Choose $0 < r \leq \tau h_0$, such that $z \in S_u(y, r)\setminus S_u(y, r/2)$. By the definition of sections and the mean value theorem, we have
	\[
		\frac{r}{2}\leq u(z) - u(y) - Du(y)\cdot (z - y) \leq \sup_{S_{h_0}} |Du(\cdot) - Du(y)|\cdot |y - z|\leq M|y-z|.
	\]
	Thus $2M|y-z| \geq r$. Then, using \eqref{osc_est},
	\begin{equation}
	\label{Hcase1}
		\begin{aligned}
			|v(y) - v(z)| \leq \underset{S_u(y, r)}{\osc} v 
			& \leq C\big( \|v\|_{L^\infty(S_u(y, \tau h_0))} + \tilde M \big)\tau^{-\gamma}h_0^{-\gamma}r^\gamma\\
			& \leq C_3 \big( \|v\|_{L^\infty(S_{h_0})} + \tilde M\big) h_0^{-\gamma}M^{\gamma}|y-z|^\gamma,
		\end{aligned}
	\end{equation}
	where $C_3 = C_3(n, \lambda, \Lambda, q,\e)$ can be chosen such that $C_3>2\tau^{-\gamma}$.

	\medskip
	\textbf{Case 2:} $z \notin S_u(y, \tau h_0)$.  As in Case 1, we have $\tau h_0 \leq M|y-z|$. Then, clearly \eqref{Hcase1} holds from 
	$|v(y) - v(z)| \leq 2 \|v\|_{L^\infty(S_{h_0})}$ and the choice of $C_3$. 
	
	Combining these two cases, we obtain \eqref{int_est_qmu} as asserted.
\end{proof}

\section{Global  H\"older estimates for linearized Monge--Amp\`ere equations}
\label{sec_glH}

In this section, we give a proof of Theorem \ref{glb_holder}.

First, we prove H\"{o}lder estimates for solutions to \eqref{eqdiv} at the boundary.
\begin{prop}\label{prop: bd_int_est}
	Assume that $u$ and $\Omega$ satisfy \eqref{glb_1}--\eqref{glb_4}. Let $\F\in W^{1, n}(\Omega; \R^n)$, $\mu\in L^n(\Omega)$, and 
	$v \in  W^{2, n}_{\text{loc}}(\Omega)\cap C(\overline{\Omega})$ be the solution to 
	\[
	U^{ij}D_{ij} v = \diver \F + \mu \quad \text{ in }  \Omega, \quad \; v = \phi \quad \text{ on } \partial \Omega,
	\]
	where $\phi \in C^\alpha(\partial \Omega)$ for some $\alpha \in (0, 1)$. Assume $(D^2u)^{1/2} \F \in L^q(\Omega;\R^n)$ for some $q > n$, and 
	there exist
	$M_0\geq 0$ and  $\e_0>0$ such that
  \[|\mu|(S_u(z, s)) \leq M_0 s^{\frac{n-2}{2} +\e_0}\quad \text{for all sections } S_u(z, s)\subset \overline{\Omega}.\]
  Let 
  \[\alpha_0:= \min \Big\{\alpha, \frac{3(q-n)}{8q}, \frac{3\e_0}{4}\Big\},\quad \alpha_1:= \frac{\alpha_0}{\alpha_0 + 3n}.\]
Then, there exist constants $\delta, C > 0$ depending only on $n, \lambda, \Lambda, \alpha, \rho$, $q$ and $\e_0$, such that 
	\[
	|v(x) - v(x_0)| \leq C \Big( \|\phi\|_{C^{\alpha}(\partial \Omega)} + \|(D^2u)^{1/2} \F\|_{L^q(\Omega)}+ M_0\Big)|x - x_0|^{\alpha_1}\,\text{if }x_0 \in \partial \Omega,\, x\in \Omega \cap B_{\delta}(x_0).
	\] 
\end{prop}
\begin{proof} The proof follows closely that of \cite[Proposition 14.32]{Le24} using maximum principles and barriers. Since our setting is different, we include the details for the reader's convenience.

	Since $\alpha_0 \leq \alpha$, we have $\|\phi\|_{C^{\alpha_0}(\partial \Omega)} \leq C_0(\alpha, \alpha_0, \rho) \|\phi\|_{C^{\alpha}(\partial \Omega)}$. 
	Thus, when proving the desired result, we can replace $\|\phi\|_{C^{\alpha}(\partial \Omega)}$ by $\|\phi\|_{C^{\alpha_0}(\partial \Omega)}$. 
	
	Let $L_u:= U^{ij}D_{ij}$. By homogeneity, we can assume that
	\[
	\mathbf{K}:= \|v\|_{L^\infty(\Omega)} + \|\phi\|_{C^{\alpha_0}(\partial \Omega)} + \|(D^2u)^{1/2} \F\|_{L^q(\Omega)} + M_0= 1.
	\]
	
	By Proposition \ref{thm: glb_mp} (ii),  
	\begin{equation*} \|v\|_{L^\infty(\Omega)}\leq C(n, \lambda,\Lambda, \rho, q,\e)\big(\|\phi\|_{C^\alpha(\partial \Omega)} + \|(D^2u)^{1/2} \F\|_{L^q(\Omega)}+ M_0\big).\end{equation*}
	
	Without loss of generality, we assume $x_0 = 0$, $u(0) = 0$, and $Du(0) = 0$. It suffices to show, using barriers, that for all $x \in \Omega \cap  B_{\delta}(0)$, 
	\[
		|v(x) - v(0)| \leq C|x|^{\alpha_1}, 
	\]
	for some structural constants  $\delta$ and $C$ depending only on $n, \lambda, \Lambda, \alpha, \rho$, $q$ and $\e_0$. 
	
\medskip	
	For any $\varepsilon \in (0, 1)$, we let
	\[
		h_\pm:= v - v(0) \pm \varepsilon \pm \frac{6}{\delta_2^3}w_{\delta_2} \; \quad \text{ in } A:= \Omega\cap B_{\delta_2}(0),
	\]
	where $\delta_2\in (0, 1)$ small is to be chosen later and the function $w_{\delta_2}$ is defined by 
	\[
		w_{\delta_2}(x) = w_{\delta_2}(x', x_n):= M_{\delta_2}x_n + u(x) - \tilde{\delta_2}|x'|^2 - \Lambda^n (\lambda \tilde{\delta_2})^{1-n} x_n^2\quad \text{ for } x=(x', x_n) \in \overline{\Omega},
	\]
	where \[\tilde{\delta_2}:= \delta_2^3/2 \quad\text{and}\quad  M_{\delta_2}:= \Lambda^n(\lambda \tilde{\delta_2})^{1-n}.\] 
	
\medskip	
	From \cite[Lemma 13.7]{Le24},  $w_{\delta_2}$ has the following properties for $\delta_2$ small:
	\begin{equation}\label{pr: glb_bd_1}
		\begin{cases}
			&U^{ij}D_{ij}w_{\delta_2} \leq -\max \{\tilde{\delta_2} \text{ trace}(U), n\Lambda\} \quad\text{ in } \Omega, \\
			&w_{\delta_2} \geq 0\quad \text{ in } \overline{\Omega\cap B_{\delta_2}(0)},\quad  \text{ and }\quad  w_{\delta_2} \geq \tilde{\delta_2} \quad\text{ on } \Omega \cap \partial B_{\delta_2}(0).
		\end{cases}
	\end{equation}
	Note that if $x\in \partial \Omega$ and $|x| \leq \delta_1(\varepsilon):= \varepsilon^{1/\alpha_0}$, then
	\[
		|v(x) - v(0)| = |\phi(x) - \phi(0)| \leq |x|^{\alpha_0} \leq \varepsilon.
	\]
	If $x\in \Omega \cap \partial B_{\delta_2}(0)$, we have $\frac{6}{\delta_2^3}w_{\delta_2}(x) \geq 3$ from \eqref{pr: glb_bd_1}. Moreover, we have
	\[
		|v(x) - v(0) \pm \varepsilon| \leq 2\|v\|_{L^\infty(A)} + \varepsilon \leq 2\mathbf{K} + \varepsilon \leq 3.
	\]
	It follows that if $\delta_2 \leq \delta_1$, then
	\[
		h_- \leq 0\quad \text{ and } h_+ \geq 0\quad \text{ on } \partial A.
	\]
	Also from \eqref{pr: glb_bd_1}, we have
	\[
				L_u h_+  \leq \diver \F +\mu,  \quad L_u h_- \geq \diver \F +\mu \quad \text{ in } A.
	\]
	By \cite[inclusion (14.59)]{Le24}, we have $A\subset S_u(0, \delta_2^{3/2})$ when $\delta_2$ is small.

\medskip	
	Now, applying the maximum principle in Proposition \ref{thm: glb_mp} (ii) to $h_+$ and $h_-$, we obtain
	\begin{equation}\label{pr: glb_bd_3}
		\begin{aligned}
			\max\{-h_+, h_-\}  \leq  C_0\big(\|(D^2u)^{1/2} \F\|_{L^q(A)} \delta_2^{\frac{3}{2} \cdot \frac{q-n}{2q}} + M\delta_2^{\frac{3}{2}\cdot \e_0}\big)
			 \leq C_1 \delta_2^{n{\tilde{\beta}}} \text{ in } A,
		\end{aligned}
	\end{equation}
	where $C_1 > 1$ depends only on $n, \lambda, \Lambda, q, \e_0$, and $\rho$, and \[\tilde{\beta} = \min\big\{\frac{3(q-n)}{4nq}, \frac{3\e_0}{2n}\big\}.\] 
	
	Next, from the boundary estimates of $u$ (see \cite[Proposition 8.23]{Le24}), we have for $\delta_2$ small
	\begin{equation}\label{pr: glb_bd_5}
	\begin{split}
		w_{\delta_2}(x) \leq M_{\delta_2}x_n + u(x) &\leq M_{\delta_2}|x| + C_2(n, \lambda, \Lambda,\rho) |x|^2 |\log |x||^2\\& \leq 2 M_{\delta_2}|x|\, \quad\quad \text{ for all  } x\in A.
		\end{split}
	\end{equation}
	Let $\delta_2 = \delta_1$. Combining \eqref{pr: glb_bd_3} and \eqref{pr: glb_bd_5} and recalling the definition of $h_+$ and $h_-$, we obtain
	\[
		\begin{aligned}
			|v(x) - v(0)|  \leq \varepsilon + \frac{6}{\delta_2^3} w_{\delta_2}(x) + C_1 \delta_2^{n\tilde{\beta}}
			& \leq \varepsilon + \frac{12}{\delta_2^3} \Lambda^n(\lambda\frac{\delta_2^3}{2})^{1-n}|x|  + C_1 \delta_2^{n\tilde{\beta}}\\
			& = \varepsilon + C_3(n, \lambda, \Lambda)\delta_2^{-3n}|x| + C_1 \delta_2^{n\tilde{\beta}}\\
			& = \varepsilon + C_3\varepsilon^{-3n/\alpha_0}|x| + C_1 \delta_2^{n\tilde{\beta}}.\\
		\end{aligned}
	\]
	If we further restrict $\varepsilon \leq C_1^{-1}$,
	then 
	\[
		C_1 \delta_2^{n\tilde{\beta}} = C_1 \delta_1^{n\tilde{\beta}} = C_1 \varepsilon^{\frac{1}{\alpha_0}n\tilde{\beta}} \leq C_1 \varepsilon^2 \leq \varepsilon.
	\]
	From the previous two estimates, we have for all $\varepsilon \leq C_1^{-1}$ and $|x| \leq \delta_1 = \varepsilon^{1/\alpha_0}$ that
	\begin{equation}\label{pr: glb_bd_6}
		|v(x) - v(0)| \leq 2\varepsilon + C_3\varepsilon^{-3n/\alpha_0}|x|.
	\end{equation}
	We may choose $\varepsilon = |x|^{\frac{\alpha_0}{\alpha_0 + 3n}}$ where $|x| \leq \delta:= C_1^{-(\alpha_0 + 3n)/\alpha_0}$ so that the conditions on $\varepsilon$ and $x$ are satisfied. Then, from \eqref{pr: glb_bd_6}, we have
	\[
		|v(x) - v(0)| \leq (2+ C_3)|x|^{\frac{\alpha_0}{\alpha_0 + 3n}} = (2+ C_3)|x|^{\alpha_1} \quad \text{in }\quad \Omega \cap B_{\delta}(0).
	\]
	 The proposition is proved.
\end{proof}

If $x_0\in \Omega$, denote by \[\bar{h}(x_0):=\sup \big\{h>0:  S_u(x_0, h) \subset \Omega\big\}\] the maximal height of all sections of $u$ centered at $x_0$ and contained in $\Omega$.
$S_u(x_0, \bar{h}(x_0))$ is called the \textit{maximal interior section of $u$ centered at $x_0$}, and it is tangent to the boundary.

 Thanks to Savin's boundary localization theorem \cite[Theorem 3.1]{S1}\label{thm: savin}, we have the following useful properties of the maximal interior sections; see \cite{S2} and \cite[Proposition 9.2]{Le24}.
\begin{prop}[Shape of maximal interior sections]
\label{prop: max_sec}  Let $u$ and $\Omega$ satisfy \eqref{glb_1}--\eqref{glb_4}. Assume that for some $y \in \Omega$, the maximal interior section $S_u(y, \bar{h}(y)) \subset \Omega$ is tangent to $\partial \Omega$ at 0; that is, $\partial S_u(y, \bar{h}(y)) \cap \partial \Omega=\{0\}$. If $h:=\bar{h}(y) \leq c$ where $c=c(n, \lambda, \Lambda, \rho)>0$ is small, then there exists a small positive constant $\kappa_0(n, \lambda, \Lambda, \rho)$ such that
\[
\begin{cases}
	D u(y)=a e_n \quad \text { for some } a \in\left[\kappa_0 h^{1 / 2}, \kappa_0^{-1} h^{1/2}\right], \\
	\kappa_0 E_h \subset S_u(y, h)-y \subset \kappa_0^{-1} E_h, \quad \text{and} \quad
	\kappa_0 h^{1 / 2} \leq \operatorname{dist}(y, \partial \Omega) \leq \kappa_0^{-1} h^{1/2},
\end{cases}
\]
where $e_n:=(0, 0,\cdots, 1 )\in\R^n$ and  the ellipsoid $E_h$ is given by
  $ E_h = A_h^{-1}B_{h^{1/2}}(0)$,
with $A_h$ being a linear transformation on $\R^n$ with the following properties:
\[
    \det A_h = 1, \quad A_hx = x - \tau_hx_n, \quad \tau_h\cdot e_n = 0, \quad \|A_h^{-1}\| + \|A_h\| \leq \kappa_0^{-1}|\log h|.
\]
\end{prop}

\medskip
Finally, we are able to prove Theorem \ref{glb_holder}. 
\begin{proof}[Proof of Theorem \ref{glb_holder}]
	Note that our assumption implies \eqref{glb_1}--\eqref{glb_4}, where $\rho$ now depends only on $n, \lambda, \Lambda, \|u\|_{C^3(\partial \Omega)}$, $R$, and the $C^3$ regularity of $\partial \Omega$; see \cite[Proposition 4.7]{Le24}. 
Therefore, 
	Proposition \ref{prop: bd_int_est}  can be applied to all $x_0 \in \partial \Omega$. Our proof is similar to that of \cite[Theorem 13.2]{Le24}. Since our setting is slightly different, we include the details for the reader's convenience.

	By Proposition \ref{thm: glb_mp} (ii),  
	\begin{equation}\label{vbound} \|v\|_{L^\infty(\Omega)}\leq C(n, \lambda,\Lambda, \rho, q,\e)\big(\|\phi\|_{C^\alpha(\partial \Omega)} + \|(D^2u)^{1/2} \F\|_{L^q(\Omega)}+ M_0\big).\end{equation}

    By homogeneity, we can assume that
	\[
		\mathbf{K}:= \|\phi\|_{C^\alpha(\partial \Omega)} + \|(D^2u)^{1/2} \F\|_{L^q(\Omega)} + M_0= 1.
	\]
	By the global H\"older gradient estimate in Lemma \ref{mu_bound}, we have
	\begin{equation}
	\label{Duglb2}
	M:=\sup_{x, y\in\overline{\Omega}}|Du(x)-Du(y)|\leq C_\star(n,\lambda,\Lambda,\rho).\end{equation}
   Let $c, \kappa_0\in (0, 1)$ be as in Proposition \ref{prop: max_sec} and let $\Omega_s:= \{x\in \Omega: \dist(x, \partial \Omega) > s\}$ for $s>0$. 

\medskip   
 \noindent      
    \textbf{Step 1:} H\"{o}lder estimates in the interior of $\Omega$.

By Proposition \ref{prop: max_sec}, 
    the maximal interior section $S_u(y, \overline{h}(y))$ of  $u$ centered at $y\in \overline \Omega_c$ satisfies 
	\begin{equation}\label{glb_H1}
		\bar h(y)\geq (c\kappa_0)^2=:c_1.
	\end{equation}
 Therefore, from \eqref{Duglb2} (see also \eqref{Ballha}), we can find $c_2 =c_2(n, \lambda, \Lambda,\rho)\in (0, c_1)$ such that 
    \begin{equation}
    \label{c2def}
   B_{c_2}(y)\subset S_u(y, c_1/8)\subset S_u(y, c_1/2)\Subset\Omega\quad\text{whenever }\quad y\in \overline \Omega_c.
    \end{equation}
From Theorem \ref{thm: int}, \eqref{vbound},  \eqref{Duglb2}, and \eqref{c2def}, we can find  $\gamma=\gamma(n, \lambda, \Lambda, q, \e) \in (0, 1)$ such that
    \begin{equation*}
        |v(z_1) - v(z_2)| \leq C_1(n, \lambda, \Lambda, \rho, q,\e)|z_1 - z_2|^\gamma\quad\text{for all } z_1, z_2 \in B_{c_2}(y).
    \end{equation*}
Repeating the above argument, we can find $c_3(n, \lambda, \Lambda,\rho)\in (0, c_2)$ such that 
        \begin{equation*}
        |v(z_1) - v(z_2)| \leq C_2(n, \lambda, \Lambda, \rho, q,\e)|z_1 - z_2|^\gamma\quad\text{for all } z_1, z_2\in B_{c_3}(y) \quad\text{where }\quad y\in\Omega_{c_2}.
            \end{equation*}
    
    \medskip
 \noindent   
	\textbf{Step 2:} H\"{o}lder estimates in a maximal interior section whose center is in $\overline{\Omega}\setminus\Omega_{c_2}$. 
	
	Let $y\in \Omega$ be such that $\dist (y, \partial \Omega) = r \leq c_2$. Let $S_u(y, \overline{h})$ be the maximal interior section of $u$ centered at $y$ and $y_0 \in \partial \Omega \cap \partial S_u(y, \bar{h})$ where $\bar h=\bar h(y)$. Then, the arguments in \eqref{glb_H1} and
	\eqref{c2def} give $\bar h\leq c$. By Proposition \ref{prop: max_sec}, 
	\begin{equation}\label{glb_H2}
		\kappa_0 \overline{h}^{1/2} \leq r \leq \kappa_0^{-1} \overline{h}^{1/2}, \text{ and } \kappa_0 E \subset S_u(y, \overline{h}) - y \subset \kappa_0^{-1} E,
	\end{equation}
	where  $E:=\overline{h}^{1/2} A_{\overline{h}}^{-1} B_1(0)$ with  $A_{\overline{h}}$ being a linear transformation on $\R^n$ such that
	\begin{equation}\label{glb_H3}
		 \|A_{\overline{h}}\| + \|A_{\overline{h}}^{-1}\| \leq \kappa_0^{-1}|\log \overline{h}| \quad \text{ and } \det A_{\overline{h}} = 1.
	\end{equation}
	From \eqref{glb_H2}, \eqref{glb_H3}, and
	\[(1/8) (S_u(y, \bar{h})-y)\subset S_u(y, \bar{h}/8)-y,\] we obtain constants $C_3, \overline{c_2} > 0$ depending only on $n, \lambda, \Lambda$ and $\rho$ such that
    \begin{equation}\label{glb_H6}
		B_{C_3r|\log r|}(y) \supset S_u(y, \overline{h}) \supset S_u(y, \overline{h}/8) \supset B_{\overline{c_2}\frac{r}{|\log r|}}(y) \supset B_{\overline{c_2} r^2}(y).
	\end{equation}

	Let $\bar v = v-v(y_0)$. Then
	\[U^{ij} D_{ij} \bar v = \diver \F + \mu\quad\text{in}\quad\Omega.\]
By Theorem \ref{thm: int}, there exists $C_4= C_4(n, \lambda, \Lambda, q, \e,\rho)$ such that for all $z_1, z_2 \in S_u(y, \bar{h}/8)$
    \begin{equation}\label{glb_H4}
    \begin{split}
      |v(z_1)-v(z_2)|&=   |\bar v(z_1)-\bar v(z_2)|\\
       &\leq C_4\big(\|\bar v\|_{L^\infty(S_u(y, \bar{h}))} + \|(D^2 u)^{1/2} \F\|_{L^q(\Omega)} {\bar h}^{\frac{q-n}{2q}}+ M_0\bar h^\e\big){\bar h}^{-\gamma} |z_1 - z_2|^{\gamma}. 
       \end{split}
       \end{equation}

        \medskip
    We need to estimate $\|\bar v\|_{L^\infty(S_u(y, \bar{h}))}$. 
     By \eqref{glb_H6} and Proposition \ref{prop: bd_int_est} applied at $y_0\in\p\Omega$, 
    \begin{equation}
    \label{barvb}
        \|\bar v\|_{L^\infty(S_u(y, \bar{h}))} \leq C_5(n, \lambda, \Lambda, \rho, \alpha, q,\e)[\diam(S_u(y, \bar{h}))]^{\alpha_1}\leq C_6(r|\log r|)^{\alpha_1},
    \end{equation}
    where $C_6 = C_6(n, \lambda, \Lambda, \rho, \alpha, q,\e) > 0$ and $\alpha_1 = \alpha_1(n, \alpha, q,\e) \in (0, \alpha)$.

    \medskip

   From   \eqref{glb_H2}, \eqref{glb_H6}--\eqref{barvb}, $q>n$, and $\mathbf{K} \leq 1$, we can find 
    $\beta = \beta(n, \lambda, \Lambda, q, \e, \alpha)\in (0, \min\{\alpha_1,\gamma\})$ such that
    \begin{equation}\label{glb_H8}
        |v(z_1) - v(z_2)|\leq C_7(n, \lambda, \Lambda, \rho, q, \e, \alpha)|z_1 - z_2|^{\beta}\quad\quad\text{for all }z_1, z_2 \in S_u(y, \bar{h}/8).
    \end{equation}
    
\medskip
 \noindent      
	\textbf{Step 3:} Completion of the proof. Let $x, y \in \Omega$. We show that 
\begin{equation}
	\label{vxyb}
		|v(x) - v(y)| \leq C_8(n, \lambda, \Lambda, \alpha, q, \e, \rho) |x - y|^{\beta}.
	\end{equation}
By  \eqref{vbound}, it suffices to consider \[|x-y|< c_3.\] Let $x_0, y_0\in \partial \Omega$ be such that $r_x:= \dist(x, \partial \Omega) = |x - x_0|$ and $r_y:= \dist(y, \partial \Omega) = |y - y_0|$. 
We may assume $r_x \leq r_y$. Let $\bar h=\bar h(y)$ be as in Step 2.

\medskip
\noindent
{\bf Case 1:} $r_y\leq c_2$. 

	If $x \in B_{\overline{c_2} r_y^2}(y)$, then we obtain \eqref{vxyb} directly from \eqref{glb_H6} and \eqref{glb_H8}.
	
	Assume now $x \notin B_{\overline{c_2}  r_y^2}(y)$. Then, $|x-y| \geq \overline{c_2}  r_y^2$, and 
	\begin{equation}\label{glb_H9}
			|x_0 - y_0| \leq |x - x_0| + |y - y_0| + |x - y| \leq 2r_y + |x - y| 
			 \leq \tilde{c_2}|x - y|^{1/2},
	\end{equation}
	where $\tilde{c_2}= \tilde{c_2}(n, \lambda, \Lambda, \rho)$.
	Note that
	\begin{equation}\label{glb_H10}
		|v(x) - v(y)| \leq |v(x) - v(x_0)| + |v(y) - v(y_0)| + |v(x_0) - v(y_0)|. 
	\end{equation}
	We can apply Proposition \ref{prop: bd_int_est} to get
	\begin{equation}\label{glb_H11}
		|v(x) - v(x_0)| \leq C_6|x - x_0|^{\alpha_1} \quad \text{ and }\quad |v(y) - v(y_0)| \leq C_6|y - y_0|^{\alpha_1}.
	\end{equation}
	Combining \eqref{glb_H9}--\eqref{glb_H11}, and recalling the $C^\alpha$ regularity of $v$ on $\partial \Omega$, we obtain
	\[
|v(x) - v(y)| \leq 2C_6r_y^{\alpha_1} + (\tilde{c_2})^{\alpha} |x-y|^{\alpha/2}
			 \leq C_8(n, \lambda, \Lambda, \alpha, q, \e, \rho) |x - y|^{\beta}.
	\]
\noindent
{\bf Case 2:} $r_y\geq c_2$.  Since $|x-y|< c_3$, \eqref{vxyb} follows from Step 1.
	
The proof of the theorem is complete.
\end{proof}

\section{Application to singular Abreu type equations}
\label{sec_Ab}

In this section, we will use Theorem \ref{glb_holder} to establish the Sobolev solvability in all dimensions for a class of singular fourth-order Abreu type equations. These equations arise from the approximation analysis of convex functionals subject to convexity constraints such as the Rochet--Chon\'e model \cite{RC} with a quadratic cost in the monopolist's problem in economics. We refer the readers to \cite{CR, KLWZ, Le_CPAM, Le23, LZ} for more details. 

\medskip
Theorem \ref{glb_holder} provides us a tool to directly establish, in all dimensions, the global higher-order estimates for singular Abreu equations that were first obtained in the work of Kim, Wang, Zhou, and the second author \cite{KLWZ}. Moreover, it allows us to obtain new solvability results in dimensions at least three that were not accessible by previous methods.

\medskip
Our solvability results state as follows.

\begin{thm}
[Solvability of the second boundary value problem for singular Abreu type equations]
\label{Abthm}
 Let $\Omega\subset\R^n$ be a smooth, uniformly convex domain.  Let $G(t): (0,\infty)\to\R$ be \[\text{either}\quad G(t) =\log t \quad \text{or}\quad G(t) =\log t/\log \log (t+ e^{e^{4n}}).\] 
Assume that $f\in L^p(\Omega)$ with $f\leq 0$ and $p>n$, $\varphi\in W^{4, p}(\Omega)$ and $\psi\in W^{2, p}(\Omega)$ with $\min_{\p \Omega}\psi>0$. 
Consider the following second boundary value problem for a uniformly convex function $u$:
\begin{equation}
\label{eq:Ab}
\left\{
  \begin{alignedat}{2}U^{ij}D_{ij}w~& =-\Delta u + f~&&\text{\ in} ~\ \ \Omega, \\\
 w~&=G' (\det D^2 u)~&&\text{\ in}~\ \ \Omega,\\\
u ~&=\varphi,\quad w ~= \psi~&&\text{\ on}~\ \ \p \Omega.
\end{alignedat}
\right.
\end{equation}
Here $(U^{ij})= (\det D^2 u)(D^2 u)^{-1}$. Then,  there exists a uniformly convex solution $u\in W^{4, p}(\Omega)$ to \eqref{eq:Ab}
with 
\begin{equation}
\label{ubdata}
\|u\|_{W^{4,p}(\Omega)}\leq C=C(n, p, G, \varphi, \psi, \Omega).
\end{equation}
\end{thm}
\medskip
Before sketching the proof of Theorem \ref{Abthm}, we make some pertinent remarks and related analysis.

\medskip
When $G(t)=\log t$, the expression $U^{ij}D_{ij} w= U^{ij}D_{ij} [(\det D^2 u)^{-1}]$ appears in the Abreu's equation \cite{Ab} in the constant scalar curvature problem in K\"ahler geometry \cite{D1}. For a convex function $u$ without further regularity, $\Delta u$ can be just a nonnegative Radon measure so it is a very singular term. For this reason, we call the first equation of \eqref{eq:Ab} singular Abreu equation. This fourth-order equation in $u$ can be rewritten as a system of two equations: one is a Monge--Amp\`ere equation for $u$ in the form of
\[\det D^2 u = G'^{-1} (w),\]
and the other is a linearized Monge--Amp\`ere equation for $w$ in the form of
\[U^{ij}D_{ij}w =-\Delta u + f.\]
Since we prescribe the boundary values of both $u$ and its Hessian determinant $\det D^2u$ via $w$, we call \eqref{eq:Ab} a second boundary value problem.

\medskip
Without the term $-\Delta u$ on the right-hand side of  \eqref{eq:Ab}, Theorem \ref{Abthm} was proved in \cite{Le_JDE} (see Theorems 1.1 and 1.2 and Remark 1.4 there) without the sign restriction on $f$. However, when $-\Delta u$ appears on  the right-hand side of  \eqref{eq:Ab}, it creates several difficulties in establishing solvability results for  \eqref{eq:Ab} using degree theory and a priori estimates.

\medskip
The first difficulty lies in obtaining the \textit{a priori} positive lower and upper bounds for $\det D^2 u$, which is a critical step in 
applying the regularity results of the linearized Monge-Amp\`ere equation. For example, it is not known if one can obtain a lower bound for $\det D^2u $ in dimensions $n\geq 3$ for $G(t) = t^{\frac{1}{n+2}}$ which arises from the affine maximal surface equation \cite{TW00}.

\medskip
The second difficulty, granted that the bounds $0<\lambda\leq\det D^2 u\leq\Lambda<\infty$ have been established, consists in obtaining H\"older estimates for $w$ in the linearized Monge-Amp\`ere equation
$U^{ij}D_{ij} w=-\Delta u + f$.  If one applies the global H\"older estimates in \cite{LN3}, this requires $\Delta u$ to be in $L^q$ with $q>n/2$.
However, $\Delta u$ is a priori at most $L^{1+\e}$ for some small constant $\e(\lambda,\Lambda, n)>0$ (see \cite{DPFS13, Sch, W95}) so it does not have enough integrability in dimensions $n\geq 3$. New ideas are required. 

\medskip
When $G(t) = \log t$, Theorem \ref{Abthm} was proved in \cite[Theorem 1.1]{KLWZ}. The key idea in \cite{KLWZ} is to transform the equation $U^{ij}D_{ij} w=-\Delta u + f$ with singular term $-\Delta u$ into {\it a family of linearized Monge--Amp\`ere equations} with bounded drifts and $L^p$ right-hand sides for which the interior Harnack inequality and H\"older estimates in \cite{Le_Harnack} are applicable and at each boundary point $x_0$ of interest, one can make the drift vanish at that point for the purpose of obtaining pointwise H\"older estimates at $x_0$. Indeed, let 
\begin{equation}
\label{twist1}\eta(x):= w(x) e^{|Du(x)-Du(x_0)|^2/2}.\end{equation}
Then $\eta$ solves
\begin{equation}
\label{twist2}U^{ij} D_{ij}\eta -(\det D^2 u)(Du-Du(x_0))\cdot D\eta = e^{|Du-Du(x_0)|^2/2} f.
\end{equation}

\medskip
Note that once the Hessian determinant bounds  have been established, the potential method in \cite{CWZ} gives new interior higher-order estimates for \eqref{eq:Ab}; see \cite[Theorem 5.1]{CWZ}. Here, we are able to handle estimates up to the boundary. This is crucial for the solvability using the degree theory.

\medskip
Interestingly, the transformations \eqref{twist1} and \eqref{twist2} are very specific to the case $G(t) = \log t$. They do not transform in any helpful ways for other $G$, including
$G(t) =\log t/\log \log (t+ e^{e^{4n}})$. This case is not accessible by the methods in \cite{KLWZ, Le_CPAM} in dimensions at least three. 
For this $G$, we can, however, use Theorem \ref{glb_holder} to obtain solvability.

\begin{proof} [Sketch of proof of Theorem \ref{Abthm}] It suffices to establish the a priori bound \eqref{ubdata} for uniformly convex solution $u\in W^{4, p}(\Omega)$ to \eqref{eq:Ab} since the existence then follows from the degree theory. 

\medskip
\noindent
{\bf Step 1:} Hessian determinant bounds. We assert that
\begin{equation}
\label{detbound}
C_1^{-1}\leq \det D^2 u\leq C_1 \quad\text{in }\Omega,\end{equation}
where $C_1=C_1(n, p, G, \varphi, \psi, \Omega) >0$.

Indeed, from the convexity of $u$, we have $\Delta u\geq 0$. Hence, recalling $f\leq 0$, we find
\[U^{ij}D_{ij}w =-\Delta u + f\leq 0 \quad\text{in }\Omega.\]
By the maximum principle, 
$w$ attains its minimum value in $\overline{\Omega}$ on the boundary. Thus \[w\geq\min_{\partial\Omega}w= \min_{\partial\Omega}\psi>0\quad \text{in }\Omega.\] 
This together with $\det D^2u=(G')^{-1} (w)$ gives the upper bound for the Hessian determinant:
\begin{equation}
\label{detuup}
\det D^2u\leq C_2:=(G')^{-1} ( \min_{\partial\Omega}\psi)\quad\text{in }\Omega.
\end{equation}
Using $u=\varphi$ on $\partial\Omega$ together with $\Omega$ being smooth and uniformly convex, 
we can construct suitable barrier functions (see \cite[Theorem 3.30]{Le24}) to obtain
\begin{equation}\label{uDub}
   \sup_{\Omega}|u| +  \|Du\|_{L^{\infty}(\Omega)}\leq C_3= C_3(n, G, \varphi, \psi, \Omega).
\end{equation}
Denote the Legendre transform $u^{\ast}$ of $u$ by $$u^{\ast}(y)=x \cdot Du(x)-u(x), \quad \text{where }y=Du(x)\in \Omega^\ast: =Du(\Omega). $$
Then
\[x= Du^* (y) \quad\text{and} \quad
D^2 u(x) = (D^2 u^*(y))^{-1}.
\]
Let $( u^{\ast ij})_{1\leq i, j\leq n}$ be the inverse matrix of $D^2u^{\ast}$ and
\begin{equation}
\label{wastdef}
\begin{split}
w^*&:= G\big((\det D^2 u^*)^{-1}\big)-(\det D^2 u^*)^{-1} G' \big(\det D^2 u^*)^{-1}\big)\\&
= G(\det D^2 u)- (\det D^2 u) G'(\det D^2 u).
\end{split}
\end{equation}
Computing as in \cite[Lemma 2.7]{Le_JDE}, we have
\[U^{ij}D_{ij} w= - u^{\ast ij} D_{ij} w^*,\]
and (see also \cite[(2.9)]{KLWZ})
\[\Delta u =  u^{\ast ii}. \]
Here, we use the notation
\[D_{ij} w^*= \frac{\p^2 w^*}{\p y_i \p y_j}.\]
It follows that
\begin{align}\label{uast_eq}
 u^{\ast ij}D_{ij} \big(w^*-|y|^2/2)\big)=- U^{ij} D_{ij} w- \Delta u= - f(x)=-f(D u^*(y)) \quad \text{in }\Omega^\ast.
\end{align}
\medskip
\noindent
Note that \eqref{uDub} implies
\begin{equation}\label{uastb}
   \text{diam}(\Omega^\ast) +  \|u^*\|_{L^{\infty}(\Omega^\ast)}\leq C_4(n, G, \varphi, \psi,  \Omega).
\end{equation}
Moreover, 
for $y= Du(x)\in \p\Omega^\ast$ where $x\in\p\Omega$, we have from \eqref{wastdef} 
\begin{equation}
\label{wastb}
\begin{split}
w^*(y)&
= G({G'}^{-1}(w(x)))- {G'}^{-1}(w(x)) w(x)\\ &= G({G'}^{-1}(\psi(x)))- {G'}^{-1}(\psi(x)) \psi(x).
\end{split}
\end{equation}
From \eqref{uastb} and \eqref{wastb}, by applying the ABP estimate \cite[Theorem 9.1]{GT} to $w^{*}-|y|^2/2$ on $\Omega^{*}$ in equation \eqref{uast_eq}, and then changing of variables $y= Du(x)$ with $dy = \det D^2 u~ dx,$ we obtain
\begin{equation*}
\begin{split}
 \|w^{\ast}-|y|^2/2\|_{L^{\infty}(\Omega^*)} &\leq \|w^{\ast}-|y|^2/2\|_{L^{\infty}(\p\Omega^*)} + C(n) \text{diam} (\Omega^*) \Big\|\frac{f(Du^{*})}{(\det {u^{*}} ^{ij})^{1/n}}\Big\|_{L^n(\Omega^*)}\\
 &\leq C + C \Big(\int_{\Omega^{*}} \frac{|f|^n(Du^{*})}{ (\det D^2 u^*)^{-1}}~ dy\Big)^{1/n}\\&=
 C + C  \Big(\int_{\Omega} \frac{|f|^n(x)}{ \det D^2 u} \det D^2 u~ dx\Big)^{1/n}  \\
 &\leq C_5 + C_5 \|f\|_{L^p(\Omega)},
 \end{split}
\end{equation*}
where $C_5=C_5(n, p, G, \varphi, \psi, \Omega)$. 

In particular, $w^*$ is bounded from below by a negative structural constant. 
Due to \eqref{wastdef},
the formula for $G$ gives a positive structural  lower bound for the Hessian determinant:
\begin{equation}
\label{detulow}
\det D^2u\geq C_6(n, p, G, \varphi, \psi,  \Omega)>0 \quad\text{in }\Omega.
\end{equation}
Combining \eqref{detuup} and \eqref{detulow}, we obtain \eqref{detbound} as asserted.

\medskip
\noindent
{\bf Step 2:} Higher-order derivative estimates.

Now, by \eqref{detbound},  \eqref{uDub}, Lemma \ref{mu_bound} and Remark \ref{n2rem}, we can apply Theorem \ref{glb_holder} with \[\F=0 \quad\text{and }\quad \mu =-\diver (Du) + f.\] We obtain some $\alpha= \alpha(n, p, G, \varphi, \psi, \Omega)\in (0, 1)$ such that
$w \in C^{\alpha}(\overline{\Omega})$ with
\[\|w\|_{C^{\alpha}(\Omega)} \leq C_7(n, p, G, \varphi, \psi, \Omega).\]
Rewriting the equation for $w$, we find
\begin{equation}
\label{uGeq}
\det D^2 u = (G')^{-1}(w)\quad \text{in}\quad \Omega,\quad u=\varphi
\quad \text{on}\quad \p\Omega,
\end{equation}
with $(G')^{-1}(w)\in C^{\alpha}(\overline{\Omega})$, 
and  $\varphi\in C^{3}(\overline{\Omega})$ because $\varphi\in W^{4,p}(\Omega)$ and $p>n$.  We obtain from the global Schauder estimates for the Monge--Amp\`ere equation  \cite[Theorem 1.1]{TW08b} that
\[\|u\|_{C^{2,\alpha}(\overline{\Omega})} \leq C_8(n, p, G, \varphi, \psi, \Omega).\] 
Thus, the first equation of (\ref{eq:Ab}) is a uniformly elliptic, second order partial differential equations in $w$
with $C^{\alpha}(\overline{\Omega})$ coefficients and $L^p$ right-hand side. Hence, from the standard $W^{2,p}$ theory for uniformly elliptic equations (see \cite[Chapter 9]{GT}), we obtain the following
 $W^{2,p}(\Omega)$ estimates
 \[\|w\|_{W^{2,p}(\Omega)} \leq C(n, p, G, \varphi, \psi, \Omega).\] 
 Now, we can differentiate and apply the standard Schauder and Calder\'on--Zygmund theories  to \eqref{uGeq} to obtain the following global 
  $W^{4, p}(\Omega)$ estimate for $u$:
\begin{equation*}
\|u\|_{W^{4,p}(\Omega)}\leq C=C(n, p, G,\varphi, \psi,\Omega).
\end{equation*}
The theorem is proved.
\end{proof}

\medskip
{\bf Acknowledgements.} The authors would like to thank Professor Bin Zhou (Peking University) and Dr. Ling Wang (Bocconi University) for their interest and helpful comments.
The authors are grateful to the referee for 
providing constructive comments that helped improve the original manuscript.

\end{document}